%% file: main.tex
\documentclass{article}
\input{setting}
\title{Two Generalized Derivative-free Methods to Solve Large Scale Nonlinear Equations with Convex Constraints}
\author[1,5]{Kabenge Hamiss}
\author[1,4]{Mohammed Alshahrani}
\author[2,3]{Mujahid N. Syed}
\affil[1]{\textit{Department of Mathematics, King Fahd University of Petroleum and
Minerals, Dhahran 31261, Saudi Arabia}}
\affil[2]{\textit{Department of Industrial \& Systems Engineering, King Fahd University of Petroleum and
Minerals, Dhahran 31261, Saudi Arabia}}
\affil[3]{\textit{Interdisciplinary Research Center for Intelligent Secure Systems, King Fahd University of Petroleum and
Minerals, Dhahran 31261, Saudi Arabia}}
\affil[4]{\textit{Interdisciplinary Research Center for Smart Mobility and Logistics, King Fahd University of Petroleum and
Minerals, Dhahran 31261, Saudi Arabia}}
\affil[5]{\textit{Department of Mathematics and Statistics, Islamic University in Uganda, Mbale 2555, Uganda}}

\begin{document}

\maketitle

\input{abs}

\vskip 0.5em
\hrulefill
\section{Introduction}
\input{intro}
\input{sec2}
\input{sec3}
\input{sec4}

\input{comps}
\input{concl}

\subsection{Data availability}
The data used to support the findings of this study are included in the article.
\subsection{Conflicts of Interest}
The authors declare that they have no conflicts of interest.
\newpage
\input{ref}

\textbf{Declaration of generative AI and AI-assisted technologies in the writing process}

During the preparation of this work, the author(s) used ChatGPT (OpenAI) in order to assist with language editing, improving clarity, and refining the structure of the manuscript. After using this tool, the author(s) reviewed and edited the content as needed and take(s) full responsibility for the content of the publication.
\end{document}

%% file: setting.tex
\usepackage{setspace} 
\usepackage{booktabs}
\usepackage{longtable}
\usepackage{adjustbox}
\usepackage{geometry}
\usepackage{rotating} 
\usepackage[english]{babel} 
\usepackage{csquotes}
\usepackage[backend=biber]{biblatex}
\usepackage{amsthm} 
\usepackage{amsmath, amssymb} 
\newcommand{\R}{\mathbb{R}} 
 
\newcommand{\keywords}[1]{\textbf{Keywords: }#1}
\newtheorem{lemma}{Lemma}

\theoremstyle{remark}
\newtheorem{remark}{Remark}
\newtheorem{theorem}{Theorem}
\usepackage{algorithm}      
\usepackage{algpseudocode}  
\usepackage{graphicx} 
\usepackage{xcolor}
\usepackage{graphicx}
\usepackage{amsmath, amsthm, amsfonts}
\usepackage{graphicx}
\usepackage{caption}
\usepackage[pdfborder={0 0 0},colorlinks=true,urlcolor=blue,citecolor=red,bookmarks=false]{hyperref}
\newcommand{\tr}{\operatorname{tr}}
\usepackage{authblk}
\usepackage{siunitx}
\usepackage{booktabs}

%% file: abs.tex
\begin{abstract}%
    In this work, we propose two derivative-free methods to address the problem of large-scale nonlinear equations with convex constraints.  These algorithms satisfy the sufficient descent condition. The search directions can be considered generalizations of the Modified Optimal Perry conjugate gradient method and the conjugate gradient projection method or the Spectral Modified Optimal Perry conjugate gradient method and the Spectral Conjugate Gradient Projection method. The global convergence of the former does not depend on the Lipschitz continuity of G. In contrast, the latter's global convergence depends on the Lipschitz continuity of G. The numerical results show the efficiency of the algorithms.
\end{abstract}

%% file: intro.tex
Let $\Gamma \subset \mathbb{R}^n$ be a nonempty, closed, and convex set and $G$ be continuous mapping such that,  
\begin{align}
    G(x) = 0. \label{eq:op}
\end{align}
where $x \in \Gamma$. Equations in \eqref{eq:op} are called a system of non-linear equations.\\
Nonlinear equations represent a significant class of problems closely related to optimization challenges, which frequently emerge in various science, technology, and industry fields. Many problems have been extensively studied, highlighting their importance in all disciplines \cite{ahookhosh2013two}. These problems are often complex, requiring sophisticated methods to solve efficiently.
The signals may be smooth or non-smooth and appear in a variety of applications, such as Bregman distances~\cite{iusem1997newton}, Monotone variational inequality problems ~\cite{he2002new, malitsky2014extragradient, zhao2001monotonicity}, Non-negative matrix factorization ~\cite{berry2007algorithms}, Phase retrieval~ \cite{candes2015phase}, Constrained neural networks~ \cite{chorowski2014learning}, Chemical equilibrium systems in thermodynamics ~ \cite{meintjes1987methodology, zeleznik1968calculation}, Financial modeling and forecasting ~ \cite{dai2020efficient}, Signal processing and signal restoration, especially in wavelet deconvolution and compressed sensing ~\cite{abubakar2020note, figueiredo2007gradient},   solid state physics, plasma and fluid mechanics \cite{chen2019gramian, gao2019mathematical}.\\ 
Given the wide range of applications, developing efficient algorithms to solve \eqref{eq:op} is crucial. Various methods have been proposed to solve these problems, with considerable research focusing on improving their accuracy and computational efficiency. Among such methods include methods involving determining the derivatives of $G(x)$ in \eqref{eq:op} such as Newton and Newton-like methods in ~\cite{iusem1997newton, nocedal1999numerical}. However, they require $G(x)$ to be not only continuous but also differentiable(smooth) which may not be the case in some instances. Further, even if $G$ is differentiable, it may be difficult to obtain derivatives of $G(x)$ in \eqref{eq:op} and this requires large memory for computation and storing. This has steered mathematicians to devise methods that do not require $G(x)$ to be differentiable. Such methods are called derivative-free methods. \\
Many derivative-free methods have been proposed stemming from Conjugate gradient methods(CGM) and Spectral conjugate gradient methods. They involve determining a descent direction for subsequent iterations to be close to the solution. The appropriate step length accompanies this. Constructing a hyperplane for the current iterate to be separated from the solution set due to convex separation. Projecting $x_k$ onto the hyperplane for the algorithm's convergence.\\
Sabi'u et al.~\cite{sabi2023modified} proposed a modified optimal Perry conjugate gradient. \\method(MOPCGM) with conjugate parameter
\begin{equation}
   \theta_k^{MOP}=\frac{(v_{k-1}-\zeta_k^*s_{k-1})^TG_k}{p_{k-1}^Tv_{k-1}} 
\end{equation}
 where $\zeta_k^*=\frac{s_{k-1}^Tw_{k-1}}{||s_{k-1}||^2}$,
 \begin{equation}
     z_{k-1}=x_{k-1}+\alpha_{k-1}p_{k-1},
 \end{equation}
\begin{equation}
    s_{k-1}=z_{k-1}-x_{k-1}
\end{equation}
and
\begin{equation}
    v_{k-1}=G_k-G_{k-1}+\tau s_{k-1}, ~~~\tau\geq 0
\end{equation}

In 2020, Zheng et al~\cite{zheng2020conjugate} proposed a Conjugate gradient projection method (CGPM) with conjugate parameter 
\begin{equation}
\label{eq:HZ}
\theta_k=\frac{G_k^Tp_{k-1}-2a_k||w_{k-1}||^2}{p_{k-1}^Tw_{k-1}}
\end{equation}
  $a_k=\frac{G_k^Tp_{k-1}}{w_{k-1}^Tp_{k-1}}$, $w_{k-1}=y_{k-1}+t_{k-1}p_{k-1}$, $t_{k-1}=\max\{1,1-\frac{p_{k-1}^Ty_{k-1}}{||p_{k-1}||^2}\}$, and  $\gamma, \tau>0$\\
In this work, we propose two novel methods that generalize the Modified Optimal Perry Conjugate Gradient Method (MOPCGM) and the Conjugate Gradient Projection Method (CGPM). The main contributions of this paper are summarized as follows: \begin{enumerate} \item[(i)] The proposed methods outperform the original ones in terms of the number of function evaluations. \item[(ii)] They converge in fewer iterations. \item[(iii)] They require less computational time to solve the problem. \item[(iv)] They exhibit improved accuracy and robustness. \end{enumerate}
The paper is arranged as follows: In the next section, we derive the generalized MOPCGM and CGPM. Section 3 details the numerical experiments and their application to compressed sensing. In the last section, we give the conclusion.

%% file: sec2.tex
\section{Generalized Derivative-Free Methods}
\label{chp:generalMOPCGM}

{In this section, we give a generalized result for the Perry conjugate gradient and Conjugate gradient projection techniques. \\
We also detail the convergence analysis of algorithms  provided the following assumptions hold:
\begin{enumerate}
    \item[(A1)] $\Gamma$ is nonempty
    \item[(A2)] G is Lipschitz continuous. Then there exists some nonzero constant $L$ such that for any $x, y\in \mathbf{R}^n$ then $||G(y)-G(x)||\leq L||y-x||$
\end{enumerate}
\subsection{Generization of MOPCGM}
From $s_{k-1}=x_k-x_{k-1}$, $g_k=\nabla G(x_k)$ and $y_{k-1}=g_k-g_{k-1}$. Also assume that $B_{k-1}$ is positive-definite, then 
\[
B_k=B_{k-1}-\frac{B_{k-1}s_{k-1}(B_{k-1}s_{k-1})^T}{s_{k-1}^TB_{k-1}s_{k-1}}+\frac{y_{k-1}y_{k-1}^T}{s_{k-1}^Ty_{k-1}}
\]
It is a rank 2 matrix update called DFP.

Now, using the Quasi-Newton  approach as used in ~\cite{ perry1978modified, sabi2023modified} we propose 
\begin{equation}
  p_k=-\Tilde{Q_k}G_k  
\end{equation}
such that
\begin{equation}\label{eq:update}
  \Tilde{Q_k}=\lambda I-\lambda\frac{y_{k-1}s_{k-1}^T}{2y_{k-1}^Ts_{k-1}}-\lambda\frac{s_{k-1}y_{k-1}^T}{2y_{k-1}^Ts_{k-1}}+t_k\frac{s_{k-1}s_{k-1}^T}{y_{k-1}^Ts_{k-1}}  
\end{equation}
$\lambda>0$ and $t_k>0$. Because $G$ is monotone, then $s_{k-1}^Ty_{k-1}> 0$ for all $x \neq x^*$. This implies that both $s_{k-1}$ and $y_{k-1}$ are nonzero vectors. Let $D$ be spanned by $\{s_{k-1}, y_{k-1} \}$ and $a$ be any vector in $\R^n$ such that $a^TD\neq 0$, then 
\[
a^T\Tilde{Q_k}a=t_k\frac{(a^Ts_{k-1})^2}{y_{k-1}^Ts_{k-1}}>0
\]
and this shows that $\Tilde{Q_k}$ is positive definite.\\
But $\Tilde{Q_k}$ is rank 2 matrix update, then eigenvalue $\lambda$ is of multiplicity $n-2$.
Now we need to determine the other two eigenvalues $\eta_k^-$ and $\eta_k^+$ since $\Tilde{Q_k}$ is full rank based on the symmetric property of $\Tilde{Q_k}$. Therefore we can find a set of vectors that are mutually orthogonal $\{u_k^i\}_{i=1}^{n-2}$ such that \[
\Tilde{Q}_ku_k^i=\lambda u_k^i, ~~~i=1, \cdots, n-2
\] 
and satisfy ${u_k^i}^TD=0$ for $i=1, \cdots, n-2$.\\
Therefore $\{u_k^i\}_{i=1}^{n-2}$ are eigenvectors of $\Tilde{Q}_k$ with the corresponding eigenvalue $\lambda$ for every $u_k^i$. We can now determine the remaining $2$ eigenvalues of $\Tilde{Q}_k$ that is $\eta_k^+$ and $\eta_k^-$. The following lemmas are crucial.
\begin{lemma}
\label{A1}
    Let $\Tilde{Q_k}$ be defined as in \eqref{eq:update}, then \[
    \tr{\Tilde{Q_k}}=\lambda (n-1)+t_k\frac{||s_{k-1}||^2}{s_{k-1}^Ty_{k-1}}\] 
\end{lemma}
\begin{proof}
    Using the linearity property of the trace of a matrix, it is the sum of all eigenvalues of a matrix. So we have 
    \[
    \tr{\Tilde{Q_k}}=\tr{\lambda I}-\tr{\lambda\frac{y_{k-1}s_{k-1}^T}{2y_{k-1}^Ts_{k-1}}}-\tr{\lambda\frac{s_{k-1}y_{k-1}^T}{2y_{k-1}^Ts_{k-1}}}+\tr{t_k\frac{s_{k-1}s_{k-1}^T}{y_{k-1}^Ts_{k-1}}}
    \]
    but $\tr{\lambda I}=n\lambda$. Let $A=y_{k-1}s_{k-1}^T$, $A^T=s_{k-1}y_{k-1}^T$ and $B=s_{k-1}s_{k-1}^T$  then $\tr{A}=\tr{A^T}=y_{k-1}^Ts_{k-1}$ and $\tr{B}=||s_{k-1}||^2$. Then 
    \[
    \tr{\Tilde{Q_k}}=\lambda(n-1)+t_k\frac{||s_{k-1}||^2}{y_{k-1}^Ts_{k-1}}
    \]
\end{proof}

\begin{lemma}
\label{A2}
   Let $\Tilde{Q_k}$ be defined as in \eqref{eq:update}, then
  \[
  \tr{\Tilde{Q}_k^T\Tilde{Q}_k}=\lambda^2( n-\frac{3}{2})+\frac{\lambda^2}{2}\frac{||s_{k-1}||^2||y_{k-1}||^2}{(s_{k-1}^Ty_{k-1})^2}+t_k^2\frac{||s_{k-1}||^4}{(s_{k-1}^Ty_{k-1})^2}\]
\end{lemma}
\begin{proof}
    From 
    \begin{multline*}
    \Tilde{Q}_k^T\Tilde{Q}_k=\left(\lambda I-\lambda\frac{y_{k-1}s_{k-1}^T}{2y_{k-1}^Ts_{k-1}}-\lambda\frac{s_{k-1}y_{k-1}^T}{2y_{k-1}^Ts_{k-1}}+t_k\frac{s_{k-1}s_{k-1}^T}{y_{k-1}^Ts_{k-1}}\right)^T\\
    \left(\lambda I
    -\lambda\frac{y_{k-1}s_{k-1}^T}{2y_{k-1}^Ts_{k-1}}-\lambda\frac{s_{k-1}y_{k-1}^T}{2y_{k-1}^Ts_{k-1}}+t_k\frac{s_{k-1}s_{k-1}^T}{y_{k-1}^Ts_{k-1}}\right)
    \end{multline*}
    
   \begin{multline*}
    \Tilde{Q}_k^T\Tilde{Q}_k=\left(\lambda I-\lambda\frac{s_{k-1}y_{k-1}^T}{2y_{k-1}^Ts_{k-1}}-\lambda\frac{y_{k-1}s_{k-1}^T}{2y_{k-1}^Ts_{k-1}}+t_k\frac{s_{k-1}s_{k-1}^T}{y_{k-1}^Ts_{k-1}}\right)\\
    \left(\lambda I
    -\lambda\frac{y_{k-1}s_{k-1}^T}{2y_{k-1}^Ts_{k-1}}-\lambda\frac{s_{k-1}y_{k-1}^T}{2y_{k-1}^Ts_{k-1}}+t_k\frac{s_{k-1}s_{k-1}^T}{y_{k-1}^Ts_{k-1}}\right)
    \end{multline*}

   \begin{multline*}
    \Tilde{Q}_k^T\Tilde{Q}_k=\lambda^2 I-\lambda^2\frac{3s_{k-1}y_{k-1}^T}{4y_{k-1}^Ts_{k-1}}-\lambda^2\frac{3y_{k-1}s_{k-1}^T}{4y_{k-1}^Ts_{k-1}}+t_k\lambda\frac{s_{k-1}s_{k-1}^T}{y_{k-1}^Ts_{k-1}}+\lambda^2\frac{||s_{k-1}||^2y_{k-1}y_{k-1}^T}{4(y_{k-1}^Ts_{k-1})^2}\\-t_k\lambda\frac{||s_{k-1}||^2y_{k-1}s_{k-1}^T}{2(y_{k-1}^Ts_{k-1})^2}+\lambda^2\frac{||y_{k-1}||^2s_{k-1}s_{k-1}^T}{4(y_{k-1}^Ts_{k-1})^2}-\lambda t_k\frac{||s_{k-1}||^2s_{k-1}y_{k-1}^T}{2(y_{k-1}^Ts_{k-1})^2}+t_k^2\frac{||s_{k-1}||^2s_{k-1}s_{k-1}^T}{(y_{k-1}^Ts_{k-1})^2}
    \end{multline*}
    we have seen that $\tr{y_{k-1}s_{k-1}^T}=y_{k-1}^Ts_{k-1}$ and $\tr{s_{k-1}s_{k-1}^T}=||s_{k-1}||^2$. Therefore
    properties of trace of matrices
    \begin{multline*}
    \tr{\Tilde{Q}_k^T\Tilde{Q}_k}=\tr{\lambda^2 I}-\tr{\lambda^2\frac{3s_{k-1}y_{k-1}^T}{4y_{k-1}^Ts_{k-1}}}-\tr{\lambda^2\frac{3y_{k-1}s_{k-1}^T}{4y_{k-1}^Ts_{k-1}}}+\tr{t_k\lambda\frac{s_{k-1}s_{k-1}^T}{y_{k-1}^Ts_{k-1}}}+\tr{\lambda^2\frac{||s_{k-1}||^2y_{k-1}y_{k-1}^T}{4(y_{k-1}^Ts_{k-1})^2}}\\-\tr{t_k\lambda\frac{||s_{k-1}||^2y_{k-1}s_{k-1}^T}{2(y_{k-1}^Ts_{k-1})^2}}+\tr{\lambda^2\frac{||y_{k-1}||^2s_{k-1}s_{k-1}^T}{4(y_{k-1}^Ts_{k-1})^2}}-\tr{\lambda t_k\frac{||s_{k-1}||^2s_{k-1}y_{k-1}^T}{2(y_{k-1}^Ts_{k-1})^2}}+\tr{t_k^2\frac{||s_{k-1}||^2s_{k-1}s_{k-1}^T}{(y_{k-1}^Ts_{k-1})^2}}
   \end{multline*}
   
Therefore

   \begin{multline*}
   \tr{\Tilde{Q}_k^T\Tilde{Q}_k}=n\lambda^2 -\lambda^2\frac{3}{4}-\lambda^2\frac{3}{4}+t_k\lambda\frac{||s_{k-1}||^2}{y_{k-1}^Ts_{k-1}}+\lambda^2\frac{||s_{k-1}||^2||y_{k-1}||^2}{4(y_{k-1}^Ts_{k-1})^2}\\-t_k\lambda\frac{||s_{k-1}||^2}{2(y_{k-1}^Ts_{k-1})}+\lambda^2\frac{||y_{k-1}||^2||s_{k-1}||^2}{4(y_{k-1}^Ts_{k-1})^2}-\lambda t_k\frac{||s_{k-1}||^2}{2(y_{k-1}^Ts_{k-1})}+t_k^2\frac{||s_{k-1}||^4}{(y_{k-1}^Ts_{k-1})^2}
   \end{multline*}

 \[
   \tr{\Tilde{Q}_k^T\Tilde{Q}_k}=\lambda^2(n-\frac{3}{2})+\lambda^2\frac{||s_{k-1}||^2||y_{k-1}||^2}{2(y_{k-1}^Ts_{k-1})^2}+t_k^2\frac{||s_{k-1}||^4}{(y_{k-1}^Ts_{k-1})^2}
   \]
   for all $\lambda$.
\end{proof}
\begin{lemma}
    \label{lemma:prodegnval}
    The product of eigenvalues $\eta^+$ and $\eta^-$ is given by $\eta^+\eta^-=\frac{\lambda^2}{4}-\frac{\lambda^2}{4}\left(\frac{||s_{k-1}||||y_{k-1}||}{s_{k-1}^Ty_{k-1}}\right)^2+\lambda t_k\frac{||s_{k-1}||^2}{s_{k-1}^Ty_{k-1}}$.
\end{lemma}
\begin{proof}
Using Lemma \ref{A1} and the usual sum of all eigenvalues as the trace of a matrix, we get
\[
\lambda (n-2)+\eta_k^-+\eta_k^+=\lambda (n-1)+t_k\frac{||s_{k-1}||^2}{s_{k-1}^Ty_{k-1}}.
\]
Also using Lemma \ref{A2}, we obtain
\[
\lambda^2 (n-2)+{\eta_k^-}^2+{\eta_k^+}^2=\lambda^2( n-\frac{3}{2})+\frac{\lambda^2}{2}\frac{||s_{k-1}||^2||y_{k-1}||^2}{(s_{k-1}^Ty_{k-1})^2}+t_k^2\frac{||s_{k-1}||^4}{(s_{k-1}^Ty_{k-1})^2}.
\]
 Implying that
 \begin{equation}
     {\eta_k^-}^2+{\eta_k^+}^2=\frac{\lambda^2}{2}+\frac{\lambda^2}{2}\frac{||s_{k-1}||^2||y_{k-1}||^2}{(s_{k-1}^Ty_{k-1})^2}+t_k^2\frac{||s_{k-1}||^4}{(s_{k-1}^Ty_{k-1})^2}.
 \end{equation}
 Let $a=\frac{||s_{k-1}||^2}{s_{k-1}^Ty_{k-1}}$ and $b=\frac{||s_{k-1}||||y_{k-1}||}{s_{k-1}^Ty_{k-1}}$, then we obtain
 \begin{equation}\label{gmop3}
     \eta_k^-+\eta_k^+=\lambda+at_k
 \end{equation} and 
 \begin{equation}\label{gmop4}
     {\eta_k^-}^2+{\eta_k^+}^2=\frac{\lambda^2}{2}+\frac{\lambda^2}{2}b^2+a^2t_k^2
 \end{equation}respectively.\\
 Now from \eqref{gmop3} and \eqref{gmop4}, we obtain
 \begin{equation}\label{gmop5}
     \eta_k^-\eta_k^+=\frac{\lambda^2}{4}-\frac{\lambda^2}{4}b^2+\lambda at_k.
 \end{equation}
 \begin{remark}
     when $\lambda=1$, then we obtain 
 \[
  \eta_k^-\eta_k^+=\frac{1}{4}-\frac{1}{4}b^2+at_k
 \]
 which is the case in {MOPCGM} ~\cite{sabi2023modified}.
 \end{remark}
\end{proof}

\begin{lemma}
\label{lemma:optimalt}    
Let $\Tilde{Q}_k$ be defined as in \eqref{eq:update}, then the optimal value of $t$ is $t_k^*=\lambda\frac{s_{k-1}^Ty_{k-1}}{||s_{k-1}||^2}$.
\end{lemma}

 From \eqref{gmop3} and \eqref{gmop5}, we obtain
 \begin{equation}
     \eta_k^2-(\lambda+at)\eta_k+\left[\frac{\lambda^2}{4}-\frac{\lambda^2}{4}b^2+at_k\lambda\right]=0.
 \end{equation}
 This means
 
 \begin{equation}
     \eta_k^{\pm}=\frac{(\lambda+at_k)\pm\sqrt{(\lambda+at)^2 - 4(\frac{\lambda^2}{4}-\frac{\lambda^2}{4}b^2+at_k\lambda)}}{2}.
 \end{equation}
 For positive definiteness of the matrix $\Tilde{Q}_k$, by applying some algebra, $t>\frac{\lambda}{4a}(b^2-1)$.\\
 Consequently, we obtain 
 \[
 t>\frac{\lambda}{4}(\frac{||y_{k-1}||^2}{s_{k-1}^Ty_{k-1}}-\frac{s_{k-1}^Ty_{k-1}}{||s_{k-1}||^2}).
 \]
 Now to obtain the optimal value of $t_k$, we minimize the square of the difference between $\eta_k^+$ and $\eta_k^-$ so that the condition number is small, then 
 \begin{equation}
   (\eta_k^+-\eta_k^-)^2 =(\lambda+at)^2 - 4(\frac{\lambda^2}{4}-\frac{\lambda^2}{4}b^2+at_k\lambda).
 \end{equation}
Therefore, the optimal value of $t$ is $\frac{\lambda }{a}$ .\\
That is 
\[
t_k^*=\lambda\frac{s_{k-1}^Ty_{k-1}}{||s_{k-1}||^2}.
\]
\begin{remark}
For $\lambda=1$, we obtain the Optimal Perry parameter $t_k^*=\frac{s_{k-1}^Ty_{k-1}}{||s_{k-1}||^2}$ obtained by Subu'i et al~\cite{sabi2023modified}.
\end{remark}

 We propose a new search direction is given by 
\begin{equation}
p_k=\begin{cases}\label{GMOP}
    -G_k & k=0\\
    -M_kG_k+\theta_k^{GMOP}p_{k-1} & k\geq 1
\end{cases}
\end{equation}
such that $M_k=\lambda +\theta_k^{GMOP}\frac{G_k^Tp_{k-1}}{||G_k||^2}$ and $\theta_k^{GMOP}=\frac{(v_{k-1}-t_k^*s_{k-1})^TG_k}{p_{k-1}^Tv_{k-1}}$ and $t_k^*=\lambda \frac{s_{k-1}^Tv_{k-1}}{||s_{k-1}||^2}$, $\lambda > 0$, $v_k=y_{k-1}+\tau s_{k-1}$ and $\tau>0$. 
\begin{remark}
    \label{rem:adaptive}
    To adaptively choose $\lambda$ depending on the problem and the current iterate, we define $\lambda_k$ as the projection of two quantities onto the interval $[\alpha_{\min}, \alpha_{\max}]$, where $\alpha_{\min} > 0$. This ensures that $\lambda_k$ remains bounded between $\alpha_{\min}$ and $\alpha_{\max}$, adjusting itself dynamically based on the progress of the optimization.

We define:
\[
    \lambda_k = \Pi_{[\alpha_{\min}, \alpha_{\max}]}\left( \frac{\|v_{k-1}\|^2}{s_{k-1}^\top v_{k-1}}, \frac{s_{k-1}^\top v_{k-1}}{\|s_{k-1}\|^2} \right),
\]

The projection operator $\Pi_{[\alpha_{\min}, \alpha_{\max}]}(\cdot)$ projects the values onto the interval $[\alpha_{\min}, \alpha_{\max}]$.\\
   $\frac{\|v_{k-1}\|^2}{s_{k-1}^\top v_{k-1}}$ measures how much the $G(x)$ changes relative to the step size. A large value indicates a significant change in gradients, suggesting that a larger $\lambda_k$ may be beneficial.\\
   $ \frac{s_{k-1}^\top v_{k-1}}{\|s_{k-1}\|^2}$
   measures the alignment between the step and the change in $G(x)$, relative to the step length. A smaller value suggests a smaller $\lambda_k$.

\begin{equation}
    \label{eq:adaptivegmopcgm}
\lambda_k = \Pi_{[\alpha_{\min}, \alpha_{\max}]}\left(\max\left(\frac{\|v_{k-1}\|^2}{s_{k-1}^\top v_{k-1}}, \frac{s_{k-1}^\top v_{k-1}}{\|s_{k-1}\|^2}\right)\right)
\end{equation}
This ensures that $\lambda_k$ is bounded within $[\alpha_{\min}, \alpha_{\max}]$, with $\alpha_{\min} > 0$ to prevent the step size from becoming too small.

- $\alpha_{\min}$ ensures that $\lambda_k$ does not become too small, which could cause slow convergence.
- $\alpha_{\max}$ limits $\lambda_k$ from growing too large, preventing instability.

This adaptive scheme allows $\lambda_k$ to tune itself according to the characteristics of the problem, without requiring manual adjustments for each new optimization problem. The dynamic nature of $\lambda_k$ ensures a balance between stability and fast progress.

\end{remark}
Using \eqref{eq:adaptivegmopcgm}, \eqref{GMOP} becomes
\begin{equation}
p_k=\begin{cases}\label{eq:adaptGMOP}
    -G_k & k=0\\
    -M_kG_k+\theta_k^{GMOP}p_{k-1} & k\geq 1
\end{cases}
\end{equation}
such that $M_k=\lambda_k +\theta_k^{GMOP}\frac{G_k^Tp_{k-1}}{||G_k||^2}$ and 
\begin{equation}
    \label{eq:optimalconjpamop}
    \theta_k^{GMOP}=\frac{(v_{k-1}-t_k^*s_{k-1})^TG_k}{p_{k-1}^Tv_{k-1}},
\end{equation}

\begin{equation}
    \label{eq:optimaltmop}
    t_k^*=\lambda_k \frac{s_{k-1}^Tv_{k-1}}{||s_{k-1}||^2},
\end{equation} and
 $\lambda_k > 0$ defined in \eqref{rem:adaptive}, $v_k=y_{k-1}+\tau s_{k-1}$ and $\tau>0$. In lemma \ref{lemm1:gmopcgm}, we verified that the $p_k$ of \eqref{GMOP} satisfies the sufficient descent condition.
\begin{lemma}
\label{lemm1:gmopcgm}
Let $\{p_k\}$ and $G_k$ be produced by algorithm \ref{alg:gmopcg}. Let $\alpha_{min} > 0$, then  $G_k^Tp_k\leq-\alpha_{min}  ||G_k||^2.$ 
\end{lemma}
\begin{proof}
    To verify this, we multiply $G_k^T$ in \eqref{GMOP} and we obtain
\[
G_k^Tp_k=-\lambda_k ||G_k||^2-\theta_k^{GMOP}\frac{G_k^Tp_{k-1}}{||G_k||^2}||G_k||^2+\theta_k^{GMOP}G_k^Tp_{k-1},
\]
\[
G_k^Tp_k=-\lambda_k ||G_k||^2, 
\]
but from \eqref{eq:adaptivegmopcgm} $\lambda_k\geq \alpha_{min}$. It implies that 
\begin{equation}
\label{eq:sufficientdescentmop}
G_k^Tp_k\leq -\alpha_{min} ||G_k||^2, 
\end{equation}
\end{proof}

From the Lemma \ref{lemm1:gmopcgm} above, the descent condition generally does not depend on $t_k$.


\newpage
\subsubsection{Generalized MOPCGM Algorithm}
\renewcommand{\baselinestretch}{0.95}
\begin{algorithm}[H]
    \caption{Generalized MOPCGM}
    \label{alg:gmopcg}
    \begin{algorithmic}[1]
        \State \textbf{Input:} Function $G$, initial guess $x_0 \in \mathbb{R}^n$, 
        $\lambda_o > 0$, tolerance $\epsilon = \text{Tol}$, 
        $\rho \in (0, 1)$, parameters $\tau$,   
        $\beta > 0$, $\eta > 0$, $\zeta > 0$, $\alpha_{min}> 0$, $\alpha_{max}> \alpha_{min}$, $\gamma, \gamma_1, \gamma_2, \gamma_3, \gamma_4 \in (0, 2) $, $0<\zeta_1\leq \zeta_2$, projection function 
        $\Pi_{\Gamma}$ on convex set $\Gamma$
        \State \textbf{Output:} Solution $x^*$
        \State \textbf{Initialization:} Set $k \gets 0$, $\lambda \gets \lambda_o$, $z_k \gets x_0$, initialize $\alpha_k$
        
        \While{$\| G_k \| > \epsilon$}
             \If{$\|G(x_{k+1})\| < \|G(x_k)\|$ }
                \State $\lambda_{k+1} \gets \lambda_k$
                \State \textbf{Else}
                {$\lambda_{k+1} \gets \Pi_{[\alpha_{min}, \alpha_{max}]}(\max(\frac{\|v_{k-1}\|^2}{s_{k-1}^Tv_{k-1}}, \frac{s_{k-1}^Tv_{k-1}}{\|s_{k-1}\|^2}))$}
            \EndIf
            \State Determine $p_k$ as in \eqref{GMOP}
            \State Adjust $\alpha_k = \max \{ \rho^i \beta \}$ such that 
            \begin{equation}
                G(x_k + \alpha_k p_k)^T p_k \leq -\zeta \alpha_k \| p_k \|^2 \Pi_{[\zeta_1, \zeta_2]}(\|G(x_k + \alpha_k p_k)\|)
                \label{eq:backtrack}
            \end{equation}
            \If{$z_k = x_k + \alpha_k p_k \in \Gamma$ and $\| G(z_k) \| < \epsilon$}
                \State $x^* \gets z_k$
                \State \textbf{break}
            \EndIf
            
            \State Compute $\mu_k \gets \frac{G(z_k)^T (x_k - z_k)}{\| G(z_k) \|^2}$
            \State Update $x_{k+1} \gets \Pi_{\Gamma} (x_k - \gamma\mu_k G(z_k))$
            \State Compute $s_{k-1}=z_{k-1}-x_{k-1} $ and $v_{k-1}=G(x_{k})-G(x_{k-1})+\tau s_{k-1}$
            \State Compute $\theta_k^{MOP} $ and $t^*$ using \eqref{eq:optimaltmop}  and \eqref{eq:optimalconjpamop} respectively. 
            \If{$\|f_{k+1}\| < \|f_k\|$}
    \State $\gamma = \min(\gamma \cdot \gamma_1, \gamma_2)$
\Else
    \State $\gamma = \max(\gamma \cdot \gamma_3, \gamma_4)$
    \State \textbf{break}
\EndIf

            \If{$\|p_k\|  \approx 0$ }
                \State $x^* \gets x_k$
                \State \textbf{break}
            \EndIf   
            \State Set $x_k \gets x_{k+1}$
        \EndWhile
        \State \textbf{Return} $x^*$
    \end{algorithmic}
\end{algorithm}
\renewcommand{\baselinestretch}{1.50}

}

\newpage

\subsubsection{Convergence Analysis of Generalized MOPCGM}
\label{GeneralizedConvergence}

{The next Lemma examines that the algorithm is well defined.
\begin{lemma}
\label{lemm2:gmopcgm}
    Let $G$ be Lipschitz continuous, then 
    \[
    \alpha_k\geq \min\{\eta, \frac{\alpha_{min}\rho}{L+\zeta\zeta_2 }\frac{||G_k||^2}{||p_k||^2}\}.
    \]
\end{lemma}
\begin{proof}
    With the line search in  Algorithm \ref{alg:gmopcg}, Let $\alpha_k$ be the optimal step length that satisfies \eqref{eq:backtrack}, then $\Tilde{\alpha}_k = \alpha_k/\rho $ violets \eqref{eq:backtrack}. So, 
\[
G(x_k+\tilde{\alpha}_kp_k)^Tp_k>- \zeta \tilde{\alpha}_k||p_k||^2\Pi_{[\zeta_1, \zeta_2]}(\|G(x_k + \Tilde{\alpha}_k p_k)\|)
\]
and  $G_k^Tp_k\leq -\alpha_{min}||G_k||^2$. 
\[
||G_k||^2\leq \frac{1}{\alpha_{min}}\left[(G(x_k+\Tilde{\alpha}_kp_k)- G_k)^Tp_k-G(x_k+\Tilde{\alpha}_kp_k)^Tp_k\right]
\]
\[
||G_k||^2\leq \frac{1}{\alpha_{min}}[(G(x_k+\Tilde{\alpha}_kp_k)- G_k)^Tp_k+\zeta \tilde{\alpha}_k||p_k||^2\zeta_2]
\]

\[
||G_k||^2\leq \frac{\alpha_k(L+\zeta_2\zeta) }{\rho\alpha_{min}}||p_k||^2
\]
This completes the proof.
\end{proof}

\begin{lemma}
\label{lemm4:gmopcgm}
     The direction $p_k$ produced by  Algorithm \ref{alg:gmopcg} meets the trust region property
    \[
    \alpha_{min} \|G(x_k)\|\leq \|p_k\|\leq \kappa \|G(x_k)\|.
    \]
\end{lemma}
\begin{proof}

From 
\[
s_{k-1}=z_{k-1}-x_{k-1}=\alpha_{k-1}p_{k-1}
\]
Also $|v_{k-1}^Ts_{k-1}|\leq (L\gamma+\tau)||s_{k-1}||^2 $,  then $|t_k^{GMOP}|= \lambda_k \frac{|v_{k-1}^Ts_{k-1}|}{||s_{k-1}||^2}\leq \alpha_{max}(L\gamma+\tau)$.

 Therefore
\[
|\theta_k^{GMOP}|=|\frac{v_{k-1}^TG_k}{p_{k-1}^Tv_{k-1}}-\lambda_k \frac{s_{k-1}^Tv_{k-1}}{||s_{k-1}||^2}\frac{s_{k-1}^TG_k}{p_{k-1}^Tv_{k-1}}|,
\]
\[
|\theta_k^{GMOP}|\leq \frac{|v_{k-1}^TG_k|}{\tau\alpha_{k-1}||p_{k-1}||^2}+ \frac{\alpha_{max}(L\gamma+\tau)||s_{k-1}||^2}{||s_{k-1}||^2}\frac{|s_{k-1}^TG_k|}{\tau\alpha_{k-1}||p_{k-1}||^2},
\]
\[
|\theta_k^{GMOP}|\leq \frac{||v_{k-1}||||G_k||}{\tau\alpha_{k-1}||p_{k-1}||^2}+ \frac{\alpha_{max}(L\gamma+\tau)||s_{k-1}||^2}{||s_{k-1}||^2}\frac{||s_{k-1}||||G_k||}{\tau\alpha_{k-1}||p_{k-1}||^2},
\]
\[
|\theta_k^{GMOP}|\leq \frac{(L\gamma+\tau)\alpha_{k-1}||p_{k-1}||||G_k||}{\tau\alpha_{k-1}||p_{k-1}||^2}+\alpha_{max} \frac{(L\gamma+\tau)||s_{k-1}||^2}{||s_{k-1}||^2}\frac{\alpha_{k-1}||p_{k-1}||||G_k||}{\tau\alpha_{k-1}||p_{k-1}||^2},
\]
\[
|\theta_k^{GMOP}|\leq \frac{(L\gamma+\tau)||G_k||}{\tau||p_{k-1}||}+\alpha_{max}(L\gamma+\tau) \frac{||G_k||}{\tau||p_{k-1}||},
\]
\begin{equation}\label{mest1}
 |\theta_k^{GMOP}|\leq (1+\alpha_{max})\frac{(L\gamma+\tau)||G_k||}{\tau||p_{k-1}||} ,  
\end{equation}
 Now 
  \[ ||p_k||\leq |\lambda_k-\theta_k^{GMOP}\frac{G_k^Tp_{k-1}}{||G_k||^2}|||G_k||+|\theta_k^{GMOP}|||p_{k-1}||, \]
 \[
 ||p_k||\leq |\lambda-\theta_k^{GMOP}|||G_k||+|\theta_k^{GMOP}\frac{G_k^Tp_{k-1}}{||G_k||^2}|||p_{k-1}||.
 \]
 But also 
 \[
 |(\lambda_k-\theta_k^{GMOP}\frac{G_k^Tp_{k-1}}{||G_k||^2})G_k+\theta_k^{GMOP}p_{k-1}|\leq \alpha_{max}||G_k||+|\theta_k^{GMOP}|\frac{G_k^Tp_{k-1}}{||G_k||^2}|||G_k||+|\theta_k^{GMOP}|||p_{k-1}||.
 \]
\[
 |(\lambda_k-\theta_k^{GMOP}\frac{G_k^Tp_{k-1}}{||G_k||^2})G_k+\theta_k^{GMOP}p_{k-1}|\leq \alpha_{max}||G_k||+|\theta_k^{GMOP}|\frac{\|G_k\|\|p_{k-1}\|}{||G_k||^2}|||G_k||+|\theta_k^{GMOP}|||p_{k-1}||.
 \]
 It implies that
\[
 \|p_k\|\leq \alpha_{max}||G_k||+|\theta_k^{GMOP}|\|p_{k-1}\|+|\theta_k^{GMOP}|||p_{k-1}||,
 \]
Therefore,
\[
 \|p_k\|\leq \alpha_{max}||G_k||+2|\theta_k^{GMOP}|\|p_{k-1}\|.
 \]
 Using \eqref{mest1}
 \[
 \|p_k\|\leq \alpha_{max}||G_k||+2(1+\alpha_{max})\frac{(L\gamma+\tau)||G_k||}{\tau||p_{k-1}||} \|p_{k-1}\|.
 \]
\[
 \|p_k\|\leq \alpha_{max}||G_k||+2(1+\alpha_{max})\frac{(L\gamma+\tau)||G_k||}{\tau} .
 \]
We finally get 
\begin{equation}
\label{eq:upperboundgmopcgm}
\|p_k\| \leq \left(\alpha_{max}+2(1+\alpha_{max})\frac{(L\gamma+\tau)}{\tau}\right)\|G_k\| .
\end{equation}
From \eqref{eq:upperboundgmopcgm}, we conclude that
\begin{equation}\label{mest4}
||p_k||\leq \kappa ||G_k||,
\end{equation}
where 
\begin{equation}
\label{eq:kaapa1}
\kappa=[\alpha_{max}+2(1+\alpha_{max})\frac{(L\gamma+\tau)}{\tau}]
\end{equation}
From lemma \ref{lemm1:gmopcgm}
\begin{equation}\label{mest5}
    \alpha_{min}||G_k||^2\leq -G_k^Tp_k\leq ||G_k||||p_k||.
\end{equation}
Therefore, we complete the proof by combining \eqref{mest4} and \eqref{mest5}.

\end{proof}

\begin{lemma}
\label{lemmA1}
    Suppose all the assumptions A1 and A2 hold, then \[\lim_{k\rightarrow \infty}\alpha_k||p_k||=0\]
\end{lemma}
    
\begin{proof}
    Beginning from the line search \eqref{eq:backtrack} $ G(z_k)^Tp_k\leq -\zeta \alpha_k||p_k||^2 \|G(z_k)\|$\\ Therefore 
    \[
    G(z_k)^T(x_k-z_k)=-\alpha_kG(x_k+\alpha_kp_k)^Tp_k
    \]
    \[
    G(z_k)^T(x_k-z_k)\geq \zeta \alpha_k^2||p_k||^2\|G(z_k)\|
    \]
     \begin{equation}\label{one}
         G(z_k)^T(x_k-z_k)\geq \zeta ||x_k-z_k||^2\|G(z_k)\|
     \end{equation}
 Now, we apply the monotonicity of $G$ and A1. Therefore there is $x^*\in \Gamma$ such that $G(x^*)=0$ \\
    \[
    G(z_k)^T(x_k-x^*)=G(z_k)^T(x_k-z_k+z_k-x^*)
    \]
    \[
    G(z_k)^T(x_k-x^*)=G(z_k)^T(x_k-z_k)+G(z_k)^T(z_k-x^*)
    \]
    \[
    G(z_k)^T(x_k-x^*)\geq G(z_k)^T(x_k-z_k)+G(x^*)^T(z_k-x^*)
    \]
     \begin{equation}\label{two}
          G(z_k)^T(x_k-x^*)\geq G(z_k)^T(x_k-z_k)
     \end{equation}
    
    This implies that
    \begin{equation}\label{three}
       G(z_k)^T(x_k-x^*)\geq \zeta ||x_k-z_k||^2 \|G(z_k)\|
    \end{equation}
    Applying the non-expansive property of the projection operator, we get
    \[
   || x_{k+1}-x^*||^2=||\Pi_{\Gamma}(x_k- \gamma\mu_kG(z_k))-x^*||^2,
    \]
    \[
   || x_{k+1}-x^*||^2\leq ||(x_k- \gamma\mu_kG(z_K))-x^*||^2,
    \]
     \[
   || x_{k+1}-x^*||^2\leq ||x_k-x^*||^2-2 \gamma\mu_kG(z_k)^T(x_k-x^*)+\gamma^2\mu_k^2||G(z_k)||^2.
    \]
   
    using \eqref{two}, we obtain
     \[
         || x_{k+1}-x^*||^2\leq ||x_k-x^*||^2- 2\gamma\mu_kG(z_k)^T(x_k-z_k)+\gamma^2\mu_k^2||G(z_k)||^2.
     \]
     But $\mu_k^2=\frac{(G(z_k)^T(x_k-z_k))^2}{||G(z_k)||^4}$
   \[
         || x_{k+1}-x^*||^2\leq ||x_k-x^*||^2- \gamma(2-\gamma)\frac{(G(z_k)^T(x_k-z_k))^2}{||G(z_k)||^2}.
     \]
     Using \eqref{one}, we obtain 
    
    \begin{equation}\label{five}
       || x_{k+1}-x^*||^2\leq ||x_k-x^*||^2- \gamma(2-\gamma)\zeta^2 ||x_k-z_k||^4.
    \end{equation}
    Implying that 
    \begin{equation}
       0\leq || x_{k+1}-x^*||^2\leq ||x_k-x^*||^2
    \end{equation}
 So, the sequence $\{\|x_k-x^*\|\}$  is bounded below and non-increasing, hence $\{x_k\}$ is convergent.

From \eqref{five} we have 
\[
   || x_{k}-x^*||^2\leq ||x_0-x^*||^2-\gamma(2-\gamma)\zeta^2\sum_{j=0}^k ||x_j-z_j||^4. 
\]
This means that 
\[
  \gamma(2-\gamma)\zeta^2\sum_{j=0}^k ||x_j-z_j||^4  \leq ||x_0-x^*||^2- || x_{k}-x^*||^2\leq ||x_0-x^*||^2,
\]
hence
\begin{equation}\label{six}
  \gamma(2-\gamma)\zeta^2\sum_{j=0}^{\infty} ||x_j-z_j||^4  \leq ||x_0-x^*||^2<\infty
\end{equation}


 This completes the proof.
\end{proof}

\begin{theorem}
\label{theorem:globalcovMOP}
    Suppose $x_k$ is generated by Algorithm \ref{alg:gmopcg}, then 
    \begin{equation}\label{eq:mgcc}
       \lim_{k\rightarrow \infty}\inf ||G_k||=0 
    \end{equation}  
\end{theorem}
\begin{proof}
    Suppose that there is some $\epsilon>0$ such that $||G_k||>\epsilon$ for all $k$. 
    Using \eqref{mest5} together with this, we obtain
    \begin{equation}
    \label{eq:condition}
    \alpha_{min}\epsilon\leq \|p_k\| ~~~~~~\forall k.
    \end{equation}
    From \eqref{eq:condition}, we have $\alpha_k\rightarrow 0$ as $k\rightarrow \infty$.\\
    Applying the line search in \eqref{eq:backtrack} there is $\bar{\alpha}_k$ (that is if $\alpha_k$ is the optimal step-size that satisfies the line search, then $\bar{\alpha}_k=\frac{\alpha_k}{\rho}$ violets the line search) such that 
    \begin{equation}
    \label{eq:violets}
    -G(x_k+\bar{\alpha}_k\zeta p_k)^Tp_k< \bar{\alpha}_k\zeta||p_k||^2\Pi_{[\zeta_1, \zeta_2]}(\|G(x_k + \bar{\alpha}_k p_k)\|)\leq \bar{\alpha}_k\zeta||p_k||^2\zeta_2
    \end{equation}
    Since $x_k$ and $p_k$ are bounded, we can select sub-sequences that converge to their accumulation points. Let $\bar{x}$ and $\bar{p}$ be the accumulation points of $x_k$ and $p_k$ respectively, then as $k$ approaches infinity and by continuity of $G$ in \eqref{eq:violets}, we have 
    \[
    -G(\bar{x})^T\bar{p}\leq 0
    \]
    Also using \eqref{eq:sufficientdescentmop} as $k$ approaches infinity, we have 
    \[
    -G(\bar{x})^T\bar{p}\geq 0
    \]

    This is for sure a contradiction.
    
    Therefore 
    \[
    \lim_{k\rightarrow \infty}\inf ||G(x_k)||=0.
    \]
    Hence global convergence. This completes the proof.
    
\end{proof}
We can provide an alternative proof that depends on the Lipschitz continuity of $G$. Assume $G$ is Lipschitz continuous.
\begin{proof}
We have two cases to consider.
\begin{enumerate}
    \item Case 1\\
    Suppose $\lim_{k\rightarrow \infty}\inf\|p_k\|=0$. Now when we apply \eqref{eq:sufficientdescentmop}, it means there is $\Omega \in \R_+$ such that  
    \[
    \|G_k\|\leq \Omega\|p_k\|, ~\forall ~k
    \]
    and taking the limits concludes \eqref{eq:mgcc}.
    \item Case II\\
    Let $\lim_{k\rightarrow \infty}\inf\|p_k\|\neq 0$.  Applying the line search in \eqref{eq:backtrack} there is $\Tilde{\alpha}_k$ (that is if $\alpha_k$ is the optimal step-size that satisfies the line search, then $\Tilde{\alpha}_k=\frac{\alpha_k}{\rho}$ violets the line search) such that 
    \begin{equation}
    \label{eq:violetsmopcgm}
    -G(x_k+\Tilde{\alpha}_k p_k)^Tp_k< \zeta\Tilde{\alpha}_k||p_k||^2\Pi_{[\zeta_1, \zeta_2]}(\|G(x_k + \Tilde{\alpha}_k p_k)\|)\leq \Tilde{\alpha}_k\zeta||p_k||^2\zeta_2
    \end{equation}
    But using sufficient descent condition \eqref{eq:mgcc}, Triangle inequality, Cauchy-Schwartz inequality and Lipschitz continuity of $G$, we have
    \[
    \alpha_{min} ||G_k||^2 \leq -G_k^Tp_k=(G(x_k+\Tilde{\alpha}_kp_k)-G_k)^Tp_k-G(x_k+\Tilde{\alpha}_kp_k)^Tp_k,
    \]
     \[
    \alpha_{min} ||G_k||^2 \leq L\Tilde{\alpha}_k||p_k||^2+\zeta\Tilde{\alpha}_k||p_k||^2\zeta_2,
    \]
     \[
    \alpha_{min} ||G_k||^2 \leq \Tilde{\alpha}_k||p_k||^2\left(L+\zeta\zeta_2\right).\]
    We arrive at
    \[
    \frac{\rho\alpha_{min}}{||p_k||(L+\zeta\zeta_2 )}\|G_k\|\leq \alpha_k||p_k||.
    \]
    Consequently, we have
    \[
     ||G_k||^2\leq \alpha_k||p_k||\frac{\alpha_{min}||p_k||(L+\zeta\zeta_2) }{\rho\alpha_{min}}.
    \]
    Taking the limits in $k$ implies that 
    \[
    \lim_{k\rightarrow \infty}\inf ||G_k||=0.
    \]
    Hence global convergence. This completes the proof.
\end{enumerate} 
\end{proof}

}

%% file: sec3.tex
\subsection{Generalization of CGPM}
The CG parameter given by Hager and Zhang (HZ)~\cite{hager2005new} is a particular case of  Dai and Liao (DL) parameter ~\cite{dai2012another}. It happens when $t=2\frac{||y_{k-1}||^2}{s_{k-1}^Ty_{k-1}}$ in \eqref{eq:HZ} which may be considered to be an adaptive type of the D-L scheme.\\
We now propose a generalized parameter related to H-Z parameter $t_k$ to be 
\begin{equation}\label{eq:newhz}
t_k=\lambda \frac{||y_{k-1}||^2}{s_{k-1}^Ty_{k-1}},
\end{equation}
such that $\lambda > 0$. The proposed search direction is 
\begin{equation}
p_k=\begin{cases}\label{GCGPM}
    -G_k & k=0\\
    -\lambda_k G_k+\theta_k^{GCGM}p_{k-1}+ \tau a_kw_{k-1}& k\geq 1
\end{cases},
\end{equation}
such that $\lambda_k > 0$,
\begin{equation}
\label{eq:optimalcgpmpara}
    \theta_k^{GCGM}=\frac{G_k^Tw_{k-1}}{p_{k-1}^Tw_{k-1}}-\lambda_k\frac{||w_{k-1}||^2}{p_{k-1}^Tw_{k-1}}\frac{G_k^Tp_{k-1}}{p_{k-1}^Tw_{k-1}}, 
\end{equation}
where $\lambda_k=\Pi_{[\alpha_{min}, \alpha_{max}]}(\max(\frac{\|w_{k-1}\|^2}{s_{k-1}^Tw_{k-1}}, \frac{s_{k-1}^Tw_{k-1}}{\|s_{k-1}\|^2}))$ as used in \cite{sabi2020two} and 
\[
\Pi_{[\alpha_{min}, \alpha_{max}]}(x)=\max(\alpha_{min}, \min(x, \alpha_{max})),
\]

$a_k=\frac{G_k^Tp_{k-1}}{w_{k-1}^Tp_{k-1}}$, $\tau >0$, $w_{k-1}=y_{k-1}+r_kp_{k-1}$, $y_{k-1}=G(x_k)-G(x_{k-1})$, $s_{k-1}=x_k-x_{k-1}$ and  $r_k=1+\max\{0,-\frac{G_k^Tp_{k-1}}{w_{k-1}^Tp_{k-1}} \}$ \\
The next lemma verifies that \eqref{GCGPM} satisfies the sufficient descent condition.

\begin{lemma}
\label{lemm1:gcgpm}
Let $\{p_k\}$ and $G(x_k)$ be produced by the algorithm \ref{algorithm:gcgpm}. Let  $0\leq \tau\leq 1$, and $\alpha_{min} \geq \frac{1+\tau}{2}$  , then  \eqref{GCGPM} meets the sufficient descent  condition $G_k^Tp_k\leq -\xi\|G_k\|^2$, where $\xi\geq 0$.

\end{lemma}

The condition holds for $k=0$.
\[
G_k^Tp_k=-\lambda_k ||G_k||^2+\theta_k^{GCGM}G_k^Tp_{k-1}+\tau a_kG_k^Tw_{k-1}.
\]
\[
G_k^Tp_k=-\lambda_k ||G_k||^2+\frac{G_k^Tw_{k-1}}{p_{k-1}^Tw_{k-1}}G_k^Tp_{k-1}-\lambda_k\frac{||w_{k-1}||^2}{p_{k-1}^Tw_{k-1}}\frac{G_k^Tp_{k-1}}{p_{k-1}^Tw_{k-1}}G_k^Tp_{k-1}+\tau \frac{G_k^Tp_{k-1}}{w_{k-1}^Tp_{k-1}}G_k^Tw_{k-1}.
\]
\[
G_k^Tp_k=-\lambda_k ||G_k||^2+(1+\tau)\frac{G_k^Tw_{k-1}}{p_{k-1}^Tw_{k-1}}G_k^Tp_{k-1}-\lambda_k\frac{||w_{k-1}||^2}{p_{k-1}^Tw_{k-1}}\frac{G_k^Tp_{k-1}}{p_{k-1}^Tw_{k-1}}G_k^Tp_{k-1}.
\]
\[
G_k^Tp_k=-\lambda_k ||G_k||^2+\frac{\sqrt{2\lambda_k}(1+\tau)}{\sqrt{2\lambda_k}}\frac{G_k^Tw_{k-1}}{p_{k-1}^Tw_{k-1}}G_k^Tp_{k-1}-\lambda_k\frac{||w_{k-1}||^2}{p_{k-1}^Tw_{k-1}}\frac{G_k^Tp_{k-1}}{p_{k-1}^Tw_{k-1}}G_k^Tp_{k-1}.
\]
\[
G_k^Tp_k=-\lambda_k ||G_k||^2+\frac{\sqrt{2\lambda_k}(1+\tau)}{\sqrt{2\lambda_k}}\frac{G_k^Tw_{k-1}G_k^Tp_{k-1}p_{k-1}^Tw_{k-1}}{(p_{k-1}^Tw_{k-1})^2}-\lambda_k\frac{||w_{k-1}||^2(G_k^Tp_{k-1})^2}{(p_{k-1}^Tw_{k-1})^2}.
\]
\[
G_k^Tp_k=-\lambda_k ||G_k||^2+\frac{\frac{(1+\tau)}{\sqrt{2\lambda_k}}G_k^T(p_{k-1}^Tw_{k-1})\sqrt{2\lambda_k}w_{k-1}(G_k^Tp_{k-1})}{(p_{k-1}^Tw_{k-1})^2}-\lambda_k\frac{||w_{k-1}||^2(G_k^Tp_{k-1})^2}{(p_{k-1}^Tw_{k-1})^2}.
\]
\[
G_k^Tp_k\leq -\lambda_k ||G_k||^2+\frac{\frac{(1+\tau)^2}{2\lambda_k}||G_k||^2(p_{k-1}^Tw_{k-1})^2+2\lambda_k
||w_{k-1}||^2(G_k^Tp_{k-1})^2}{2(p_{k-1}^Tw_{k-1})^2}-\lambda_k\frac{||w_{k-1}||^2(G_k^Tp_{k-1})^2}{(p_{k-1}^Tw_{k-1})^2}.
\]
\[
G_k^Tp_k\leq -\lambda_k ||G_k||^2+\frac{(1+\tau)^2}{4\lambda_k}||G_k||^2.
\]
But $0<\alpha_{min}\leq\lambda_k\leq \alpha_{max}$, then 
\begin{equation}
\label{eq:cgpmsuffientcondition}
    G_k^Tp_k\leq -\alpha_{min}(1-\frac{(1+\tau)^2}{4\alpha_{min}^2})||G_k||^2.
\end{equation}

This completes the proof.
\newpage
\subsubsection{Generalized CGPM Algorithm}
{\renewcommand{\baselinestretch}{0.95}
{\begin{algorithm}[H]
    \caption{Generalized CGPM}
    \label{algorithm:gcgpm}
    \begin{algorithmic}[1]
        \State \textbf{Input:} Function $G$, initial guess $x_0 \in \mathbb{R}^n$, 
        $\lambda_o > 0$, tolerance $\epsilon = \text{Tol}$, 
        $\rho \in (0, 1)$, parameters $\tau > 0$, $\eta > 0$, $\zeta > 0$, $\gamma, \gamma_1, \gamma_2, \gamma_3, \gamma_4 \in (0, 2) $, $0<\alpha_{min}\leq \alpha_{max}$, $0<\zeta_1\leq \zeta_2$,  projection function 
        $\Pi_{\Gamma}$ on convex set $\Gamma$
        \State \textbf{Output:} Solution $x^*$
        \State \textbf{Initialization:} Set $k \gets 0$, $z_k \gets x_0$, $\lambda \gets \lambda_o$, initialize $\alpha_k$
        
        \While{$\| G_k \| > \epsilon$}
        \If{$\|G(x_{k+1})\| < \|G(x_k)\|$ }
                \State $\lambda \gets \lambda$
                \State \textbf{Else}
                {$\lambda \gets \Pi_{[\alpha_{min}, \alpha_{max}]}(\max(\frac{\|w_{k-1}\|^2}{s_{k-1}^Tw_{k-1}}, \frac{s_{k-1}^Tw_{k-1}}{\|s_{k-1}\|^2}))$}
            \EndIf
            \State Compute $\theta_k^{GCGPM} $ using \eqref{eq:optimalcgpmpara}.
            \State Determine $p_k$ as in \eqref{GCGPM}
            \State Adjust $\alpha_k = \max \{ \rho^i \eta \mid i = 0, 1, 2, \dots \}$ such that 
            \begin{equation}
                G(x_k + \alpha_k p_k)^T p_k \leq -\zeta \alpha_k \| p_k \|^2 \Pi_{[\zeta_1, \zeta_2]}(\| G(x_k + \alpha_k p_k) \|)
                \label{eq:backtrack2}
            \end{equation}
            \If{$z_k = x_k + \alpha_k p_k \in \Gamma$ and $\| G(z_k) \| < \epsilon$}
                \State $x^* \gets z_k$
                \State \textbf{break}
            \EndIf
            \State Compute $\mu_k \gets \frac{G(z_k)^T (x_k - z_k)}{\| G(z_k) \|^2}$
            \State Update $x_{k+1} \gets \Pi_{\Gamma}(x_k - \gamma \mu_k G(z_k))$
            \State Compute $v_{k-1}=G(x_k)-G(x_{k-1})+r_kp_{k-1}$
             \If{$\|f_{k+1}\| < \|f_k\|$}
    \State $\gamma = \min(\gamma \cdot \gamma_1, \gamma_2)$
\Else
    \State $\gamma = \max(\gamma \cdot \gamma_3, \gamma_4)$
    \State \textbf{break}
\EndIf
            \If{$\|p_k\|  \approx 0$ }
                \State $x^* \gets x_k$
                \State \textbf{break}
            \EndIf
            \State Set $x_k \gets x_{k+1}$
        \EndWhile
        \State \textbf{Return} $x^*$
    \end{algorithmic}
\end{algorithm}
\renewcommand{\baselinestretch}{1.50}
}
\newpage
\subsubsection{Convergence analysis for the generalized CGPM}
\label{Convergence GCGPM}

{Assuming all the conditions A1 and A2 are met, then we can analyze the convergence of algorithm 1.
\begin{lemma}
    Let $G$ be Lipschitz continuous. Let  $0\leq \tau\leq 1$, and $\alpha_{min} \geq \frac{1+\tau}{2}$, Then 
    \[
    \alpha_k\geq \min\{\eta, \frac{\rho[4\alpha_{min}^2-(1+\tau)^2]}{4\alpha_{min}[L+\zeta\zeta_2]}\frac{||G_k||^2}{||p_k||^2}\}
    \]
\end{lemma}
\begin{proof}
   Let $\alpha_k$ be the optimal step length that satisfies \eqref{eq:backtrack2}, then, $\Tilde{\alpha}_k = \alpha_k/\rho $ violets \eqref{eq:backtrack2}. Therefore, 
\[
G(x_k+\tilde{\alpha}_kp_k)^Tp_k>- \zeta \tilde{\alpha}_k||p_k||^2\Pi_{[\zeta_1, \zeta_2]}(||G(x_k+\tilde{\alpha}_kp_k)||)
\]
But $G_k^Tp_k\leq -\alpha_{min}[1-\frac{(1+\tau)^2}{4\alpha_{min}^2}]||G_k||^2$. So
\[
||G_k||^2\leq -\frac{4\alpha_{min}}{[4\alpha_{min}^2-(1+\tau)^2]}G_k^Tp_k
\]
\[
||G_k||^2\leq \frac{4\alpha_{min}}{\left[4\alpha_{min}^2-(1+\tau)^2\right]}[(G(x_k+\Tilde{\alpha}_kp_k)- G_k)^Tp_k-G(x_k+\Tilde{\alpha}_kp_k)^Tp_k]
\]
\[
||G_k||^2\leq \frac{4\alpha_{min}}{\left[4\alpha_{min}^2-(1+\tau)^2\right]}[(G(x_k+\Tilde{\alpha}_kp_k)- G_k)^Tp_k+\zeta \tilde{\alpha}_k||p_k||^2\Pi_{[\zeta_1, \zeta_2]}(\|G(x_k+\tilde{\alpha}_kp_k)\|)],
\]

\[
||G_k||^2\leq \frac{4\alpha_k\alpha_{min}}{\rho[4\alpha_{min}^2-(1+\tau)^2]}[L+\zeta\zeta_2]||p_k||^2.
\]

\end{proof}
\begin{lemma}
\label{lemmacgpm}
    Suppose all the assumptions A1 and A2 hold, then \[\lim_{k\rightarrow \infty}\alpha_k||p_k||=0\]
\end{lemma}
    
\begin{proof}
    Beginning from the line search \eqref{eq:backtrack2} $ G(z_k)^Tp_k\leq -\zeta \alpha_k||p_k||^2||G(z_k)|| $\\ Therefore 
    \[
    G(z_k)^T(x_k-z_k)=-\alpha_kG(x_k+\alpha_kp_k)^Tp_k
    \]
    \[
    G(z_k)^T(x_k-z_k)\geq \zeta \alpha_k^2||p_k||^2||G(z_k)||
    \]
     \begin{equation}\label{cpart1}
         G(z_k)^T(x_k-z_k)\geq \zeta ||x_k-z_k||^2||G(z_k)||
     \end{equation}
 Now, we apply the monotonicity of $G$ and A1. Therefore there is $x^*\in \Gamma$ such that $G(x^*)=0$ \\
    \[
    G(z_k)^T(x_k-x^*)=G(z_k)^T(x_k-z_k+z_k-x^*)
    \]
    \[
    G(z_k)^T(x_k-x^*)=G(z_k)^T(x_k-z_k)+G(z_k)^T(z_k-x^*)
    \]
    \[
    G(z_k)^T(x_k-x^*)\geq G(z_k)^T(x_k-z_k)+G(x^*)^T(z_k-x^*)
    \]
     \begin{equation}\label{cpart2}
          G(z_k)^T(x_k-x^*)\geq G(z_k)^T(x_k-z_k)
     \end{equation}
    
    This implies that
    \begin{equation}
       G(z_k)^T(x_k-x^*)\geq \zeta ||x_k-z_k||^2||G(z_k)|| 
    \end{equation}
    Applying the non-expansive property of the projection operator, we get
    \[
   || x_{k+1}-x^*||^2=||\Pi_{\Gamma}(x_k- \gamma\mu_kG(z_k))-x^*||^2
    \]
    \[
   || x_{k+1}-x^*||^2\leq ||(x_k- \gamma\mu_kG(z_K))-x^*||^2
    \]
     \[
   || x_{k+1}-x^*||^2\leq ||x_k-x^*||^2-2 \gamma\mu_kG(z_k)^T(x_k-x^*)+\gamma^2\mu_k^2||G(z_k)||^2
    \]
     \begin{equation}\label{cconv1}
         || x_{k+1}-x^*||^2\leq ||x_k-x^*||^2- 2\gamma\mu_kG(z_k)^T(x_k-z_k)+\gamma^2\mu_k^2||G(z_k)||^2
     \end{equation}

     \begin{equation}
         || x_{k+1}-x^*||^2\leq ||x_k-x^*||^2- 2\gamma\frac{(G(z_k)^T(x_k-z_k))^2}{\|G(z_k)\|^2}+\gamma^2\frac{(G(z_k)^T(x_k-z_k))^2}{\|G(z_k)\|^2}
     \end{equation}
   
    \[
   || x_{k+1}-x^*||^2\leq ||x_k-x^*||^2-\gamma(2-\gamma) \frac{(G(z_k)^T(x_k-z_k))^2}{||G(z_k)||^2}
    \]
    \begin{equation}\label{cconv2}
        || x_{k+1}-x^*||^2\leq ||x_k-x^*||^2- \gamma(2-\gamma)(\zeta ||x_k-z_k||^4)
    \end{equation}
   
Inequalities \eqref{cconv1} and \eqref{cconv2} follow from \eqref{cpart1} and \eqref{cpart2}  respectively.\\
From \eqref{cconv2}, we obtain
\[
0\leq || x_{k+1}-x^*||^2\leq ||x_k-x^*||^2
\]
This implies that $\{||x_0-x^*||\}$ is a decreasing sequence and bounded below. This means that $\{x_k\}$ is convergent.
From \eqref{cconv2} we settle that 
\begin{equation}\label{cconv3}
   || x_{k}-x^*||^2\leq ||x_0-x^*||^2- \gamma(2-\gamma)\zeta^2\sum_{j=0}^k( ||x_j-z_j||^4.
\end{equation}

 \[
 \sum_{k=0}^{\infty} ||x_k-z_k||^4\leq ||x_0-x^*||^2 <\infty.
 \]
 This completes the proof.
\end{proof}
    

\begin{theorem}
\label{theorem:globalconvcgpm}
    Suppose $x_k$ is produced by  Algorithm \ref{algorithm:gcgpm} and $\alpha_{min}> \frac{(1+\tau)}{2}$ as defined in algorithm \ref{algorithm:gcgpm}, then 
    \begin{equation}\label{cgcc}
       \lim_{k\rightarrow \infty}\inf ||G_k||=0 .
    \end{equation}  
\end{theorem}
\begin{proof}
We have two cases to consider.
\begin{enumerate}
    \item Case 1\\
    Suppose $\lim_{k\rightarrow \infty}\inf\|p_k\|=0$. Now when we apply \eqref{eq:cgpmsuffientcondition}, it means there is $\Omega \in \R_+$ such that  
    \[
    \|G_k\|\leq \Omega\|p_k\|, ~\forall ~k
    \]
    and taking the limits concludes \eqref{cgcc}.
    \item Case II\\
    Let $\lim_{k\rightarrow \infty}\inf\|p_k\|\neq 0$.  Applying the line search in \eqref{eq:backtrack2} there is $\Tilde{\alpha}_k$ (that is if $\alpha_k$ is the optimal step-size that satisfies the line search, then $\Tilde{\alpha}_k=\frac{\alpha_k}{\rho}$ violets the line search) such that 
    \begin{equation}
    \label{eq:violetscgpm}
    -G(x_k+\Tilde{\alpha}_k p_k)^Tp_k< \zeta\Tilde{\alpha}_k||p_k||^2\Pi_{[\zeta_1, \zeta_2]}(\|G(x_k + \Tilde{\alpha}_k p_k)\|)\leq \Tilde{\alpha}_k\zeta||p_k||^2\zeta_2
    \end{equation}
    But using sufficient descent condition \eqref{cgcc}, Triangle inequality, Cauchy-Schwartz inequality and Lipschitz continuity of $G$, we have
    \[
    \alpha_{min}[1-\frac{(1+\tau)^2}{4\alpha_{min}^2}] ||G_k||^2 \leq -G_k^Tp_k=(G(x_k+\Tilde{\alpha}_kp_k)-G_k)^Tp_k-G(x_k+\Tilde{\alpha}_kp_k)^Tp_k,
    \]
     \[
    \alpha_{min}[1-\frac{(1+\tau)^2}{4\alpha_{min}^2}] ||G_k||^2 \leq L\Tilde{\alpha}_k||p_k||^2+\zeta\Tilde{\alpha}_k||p_k||^2\zeta_2,
    \]
     \[
    \alpha_{min}[1-\frac{(1+\tau)^2}{4\alpha_{min}^2}] ||G_k||^2 \leq \Tilde{\alpha}_k||p_k||^2\left(L+\zeta\zeta_2\right).\]
    We arrive at
    \[
    \frac{\rho [4\alpha_{min}^2-(1+\tau)^2] }{4||p_k||\alpha_{min}(L+\zeta\zeta_2 )}\|G_k\|\leq \alpha_k||p_k||.
    \]
    Consequently, we have
    \[
     ||G_k||^2\leq \alpha_k||p_k||\frac{4\alpha_{min}||p_k||(L+\zeta\zeta_2) }{\rho[4\alpha_{min}^2-(1+\tau)^2]}.
    \]
    Taking the limits in $k$ implies that 
    \[
    \lim_{k\rightarrow \infty}\inf ||G_k||=0.
    \]
    Hence global convergence. This completes the proof.
\end{enumerate} 
\end{proof}
}

%% file: sec4.tex
\section{Numerical Experiments}
\label{testing}
{In this section, we studied and compared the performance of the GCGPM and GMOPCGM with the other three methods. That is STTDFPM [\cite{ibrahim2023two}, Algorithm 1], MOPCGM [\cite{sabi2023modified}, Algorithm 2.1], and CGPM [\cite{zheng2020conjugate}, Algorithm 2.1] without changing their line searches or parameters.\\  The comparison was conducted by considering the number of iterations, CPU time, and function evaluation number.\\ The 
 Algorithms were applied to Julia and tested with the same initial values.


 The number of iterations, function evaluations, and CPU time were represented by IT, FE, and CPU, respectively, and the norm of $G$ was also recorded to notice the accuracy of the algorithms. 19 test problems from different sources were used. All the experiments were conducted with three distinct dimensions, that is $10^3$, $10^4$, and $5\times 10^4$. 14 different initial points were used for each dimension in the experiment.\\
 We carefully selected the best parameters that seemed to suit each proposed algorithm for better performance.

    \begin{figure}[h] 
    \centering
    \includegraphics[width=0.8\textwidth]{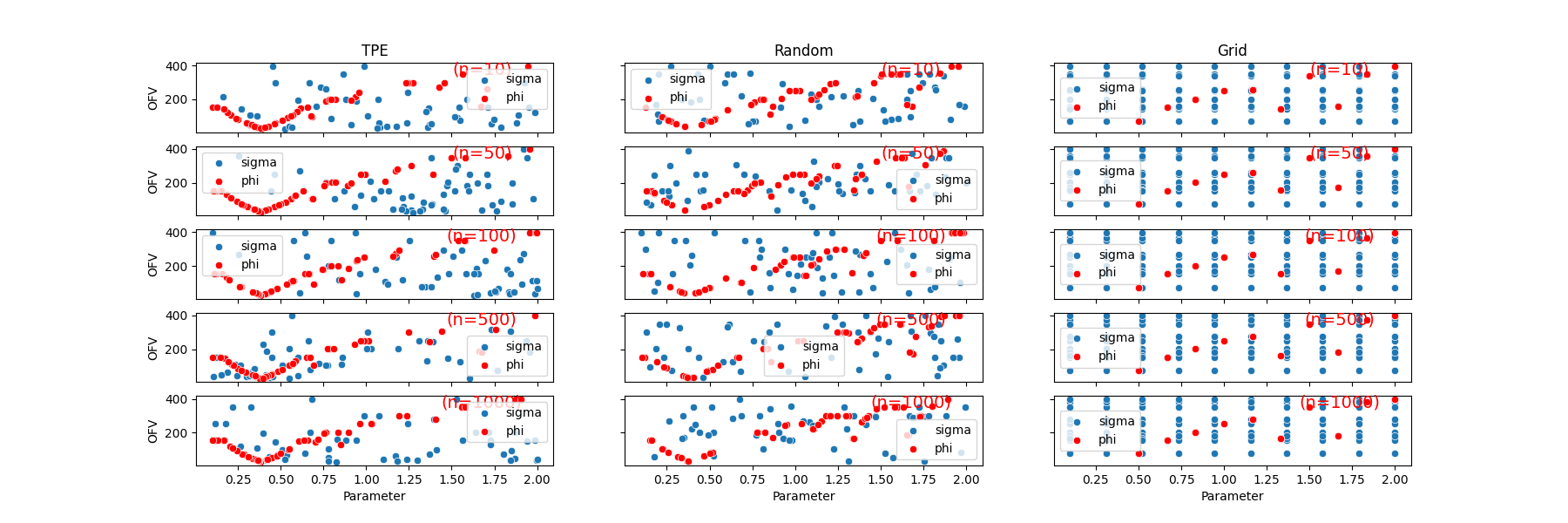} 
    \caption{Scatter diagrams for parameter values for GMOPCGM}
    \label{fig:scater diagram1} 
\end{figure}

    \begin{figure}[h] 
    \centering
    \includegraphics[width=0.8\textwidth]{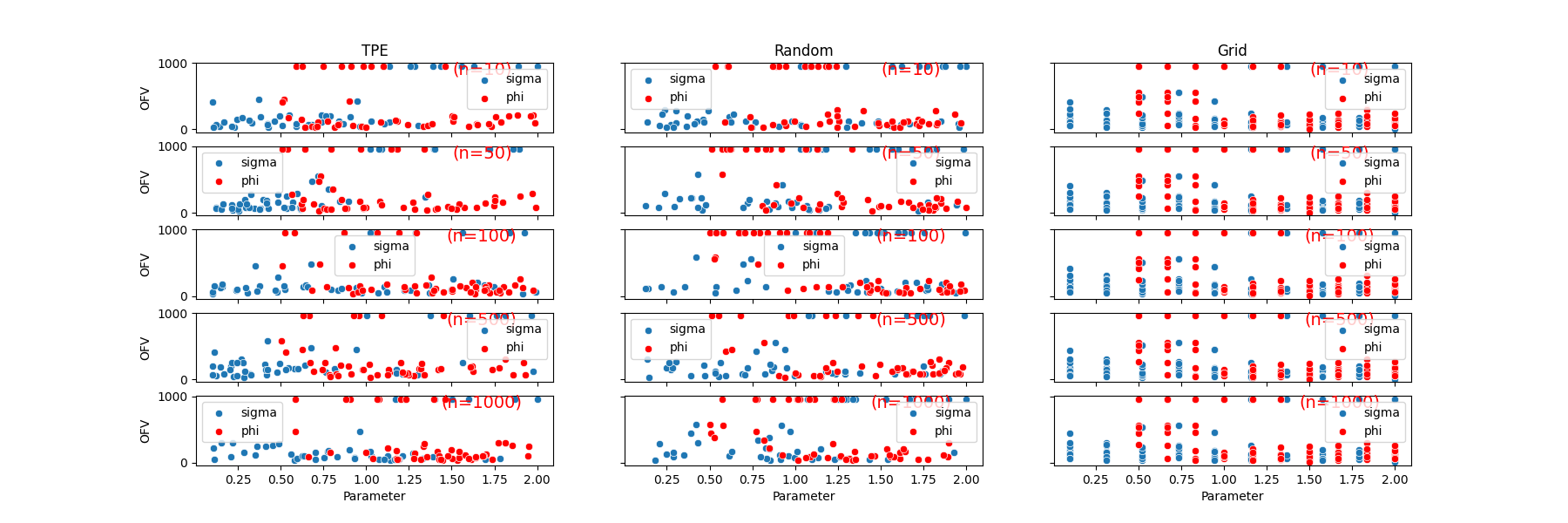} 
    \caption{Scatter diagrams for parameter values for GCGPM}
    \label{fig:scater diagram2} 
\end{figure}
In order to study the effect of parameters $\tau$ and $\phi$, we utilized machine learning-based hyperparameter selection approaches. Specifically, we have utilized the Bayesian Optimization (BO) technique to explore and exploit the parameter space. The key idea in BO is to learn a surrogate model that captures the optimization approach. We have considered a Tree-structured Parzen Estimator (TPE) approach, where the kernel density is estimated to learn dynamics between the parameter space and objective function. In addition to that, we have used non-surrogate approaches like Random Sampling (RS) and Grid Search (GS) to identify the values of the two parameters to explore and evaluate.
 
In all three approaches, both parameters are searched in the following intervals: $[0.1,10]$ for GMOPCGM and GCGPM. Moreover, we have performed 50 iterations of the two-parameter evaluations, where each iteration involves a combination of the two parameters. The study was carried out on different objective functions. Within each objective function, over multiple dimensions, we analyzed the relationship between the objective function value and the parameter values.   
 
For example, in Figure \ref{fig:scater diagram1}, results of the study on F14 using GMOPCGM are depicted. The three columns of scatter plots correspond to TPE, random, and grid search-based sampling, respectively. In a given column, the five scatter plots corresponding to $n=10,50, 100, 500, 1000$, respectively, are illustrated. Similarly, in Figure \ref{fig:scater diagram2}, results of the study on F15 using GCGPM are presented. Based on similar studies conducted over different functions w.r.t GMOPCGM \& GCGPM, we have observed the following:
\begin{enumerate}
	\item GMOPCGM
	\begin{enumerate}
		\item[(i)]$\tau$ did not have much impact on the algorithm's performance. Due to this, we selected $\tau=1.0$
		\item[(ii)] In almost all the cases, the performance of the algorithm improved when $\lambda_o=\phi$ was around $1.5$. Because of this, we decided $\lambda_o$ to take the values in $(0, 2)$.
	\end{enumerate}
	\item GCGPM
	\begin{enumerate}
		\item[(i)] $\tau$ significantly improved the algorithm's performance when it was near zero for almost all the test problems. This led us to choose $\tau=0.001$.
		\item[(ii)] The performance of the algorithm was further improved when $\lambda_o=\phi$ was greater than $1$. Due to this, we decided $\lambda_o$ to take the values in $(0.5, 2]$.
	\end{enumerate}
\end{enumerate}
 For GCGPM, we selected the following parameters
$\tau=0.001$, $\eta=0.6$, $\lambda_o=1.0$, $\rho=0.5$, $\zeta=0.1$, $\zeta_1=1.0$, $\zeta_2=1.0$, $\alpha_{min}=0.55$, $\alpha_{max}=4.9$, $\gamma=1.8$, 
$\gamma_1=1.1$, $\gamma_2=1.7$,   $\gamma_3=1.05$, $\gamma_4=1.05$\\
 For GMOPCGM, we selected the following parameters\\
 $\tau=1.0$, $\rho=0.8$,  $\beta=0.5$,  $\zeta=0.0001$, $\alpha_{min}=0.1$, $\alpha_{max}=2.0$, $\lambda_o=1.0$, 
 $\gamma=1.1$, $\gamma_2=1.8$, $\gamma_3=1.0$, $\gamma_4=1.0$, $\zeta_1=1.0$, $\zeta_2=1.0$\\
 The following are the initial points used for the experiments;\\
 $0=(0,\cdots, 0)^T$, $0.2=(8.2, \cdots, 8.2)^T$, $0.4=(\frac{2}{5},\cdots, \frac{2}{5})^T$, $0.5=(\frac{1}{2}, \cdots, \frac{1}{2})^T$, $0.6=(\frac{3}{5}, \cdots, \frac{3}{5})^T$, $0.8=(\frac{4}{5}, \cdots, \frac{4}{5})^T$, $1.0=(1.0, \cdots, 1.0)^T$, $1.1=(\frac{11}{10}, \cdots, \frac{11}{10})^T$, $1-1/m=(1-\frac{1}{m}, \cdots, 1-\frac{1}{m})^T$, $1/m=(1, \frac{1}{2}, \frac{1}{3}, \cdots, \frac{1}{m})^T$, $(k-1)/m=(0, \frac{1}{m}, \frac{2}{m}, \cdots, \frac{m-1}{m})^T$, $1/m=(\frac{1}{m}, \frac{1}{m}, \cdots, \frac{1}{m})^T$, $1/3^k=(\frac{1}{3}, \frac{1}{3^2}, \cdots, \frac{1}{3^m})^T$, $k/m=(\frac{1}{m}, \frac{2}{m}, \cdots, 1)^T$.\\
 All algorithms were terminated if either of the following was met;
 \begin{enumerate}
     \item $\|G(x_k)\|<Tol$ \item $ p_k<0.1 Tol$ \item  $k>2000$
 \end{enumerate}
  $\|G(x_k)\|$  represents the usual norm of $G$ at $x_k$ and $Tol=10^{-11}$. \\
  For a better assessment and comparison of the performance of the various schemes, we employed the  Moré and Dolan performance profile in ~\cite{dolan2002benchmarking}. The performance profile of the three new algorithms and their counterparts are represented in Figures \ref{fig:functionevals}-\ref{fig:time}. As in the figures, the vertical axes depict the chances that a certain solver outperforms the rest of the competing algorithms.
\newpage
    \begin{figure}[h] 
    \centering
    \includegraphics[width=0.8\textwidth]{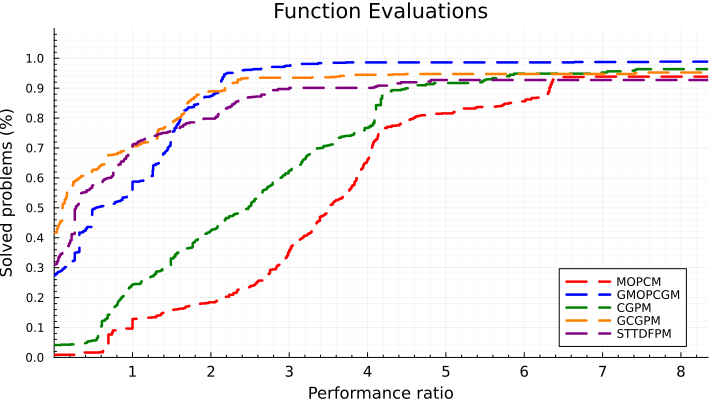} 
    \caption{Profile of function evaluations}
    \label{fig:functionevals} 
\end{figure}

    \begin{figure}[h] 
    \centering
    \includegraphics[width=0.8\textwidth]{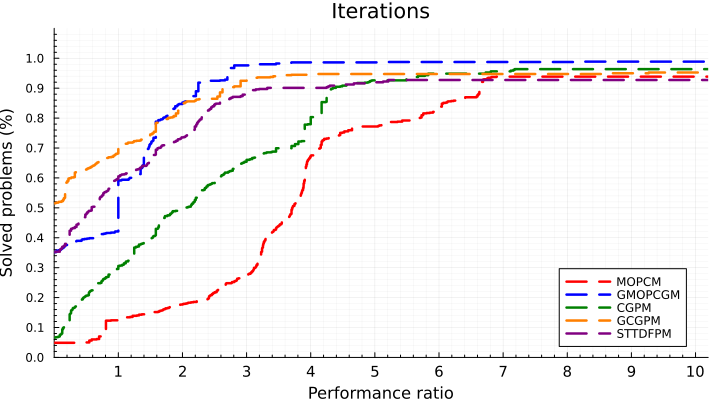} 
    \caption{Profile of iterations}
    \label{fig:iterations} 
\end{figure}
\newpage
    \begin{figure}[h] 
    \centering
    \includegraphics[width=0.8\textwidth]{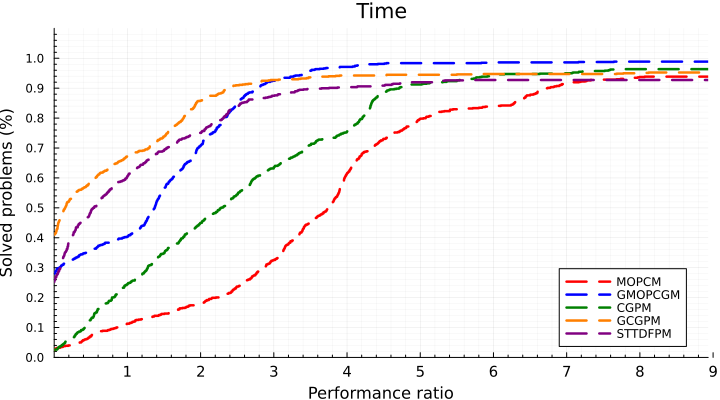} 
    \caption{Profile of time}
    \label{fig:time} 
\end{figure}
From Figures \ref{fig:functionevals}-\ref{fig:time}, it was observed that GCGPM outperformed its counterparts where it won by  $42\%$ $51\%$, and $41\%$ as depicted in Figures \ref{fig:functionevals}, \ref{fig:iterations} and \ref{fig:time} respectively. This is because the four methods, GCGPM, STTDFPM, and GMOPCGM, share the same properties, so this could likely cause a small deviation in their performance. This can be witnessed from the figures. However, There was a great improvement of the two generalized methods GCGPM and GMOPCGM compared to CGPM and MOPCGM respectively. In addition the GMOPCGM and STTDFPM were to close to each other in terms of Iterations and CPU time though the later slightly lagged behind in function evaluations. \\
The following are the problems used in the experiment.\\
{\bf Problem 5.1}(Problem 4.1 in ~\cite{sabi2020two}) $G(x)$ is \\
 $G(x_i) = 2x_i - \sin x_i, ~~i = 1, 2,\cdots, n$ and $\Gamma = [-2, \infty].$\\
{\bf Problem 5.2}(Problem 10 in ~\cite{la2006spectral}) \\
 $G(x_i) = \log(|x_i| + 1) - \frac{x_i}{n}, ~~
 , i = 1, 2, \cdots, n.$ and $\Gamma = \R.$\\
 {\bf Problem 5.3}(Problem 4.1 in\cite{zheng2020conjugate})\\
 $G(x_i) = \exp(x_i) - 1, ~~
 , i = 1, 2, \cdots, n.$ and $\Gamma = \R.$\\
 {\bf Problem 5.4}(Problem 4.5 ~\cite{sabi2020two} )The general interpretation $G(x)$ defined as\\
 $G(x_i) = 4x_i+(x_{i+1}-2x_i)-\frac{x_{i-1}^2}{3}, ~~
 , i = 1, 2, \cdots, n-1.\\
 G(x_n)=4x_n+(x_{n-1}-2x_n)-\frac{x_{n-1}^2}{3}$ \\and $\Gamma = \R.$\\
 {\bf Problem 4.5}(Problem 4.4 in \cite{zheng2020conjugate}) Exponential problem $G(x)$ defined as\\
 $G(x_1) = x_1-\exp{\cos(\frac{x_1+x_2}{n+1})},\\
 G(x_i) = x_i-\exp{\cos(\frac{x_{i-1}+x_i+x_{i+1}}{n+1})}, ~~ i = 2, \cdots, n-1$.\\
 $ G(x_n) = x_n-\exp{\cos(\frac{x_{n-1}+x_n}{n+1})}$\\
 and $\Gamma =\R$\\
{\bf Problem 5.6}(Problem 4.4~\cite{zheng2020conjugate}) Exponential problem $G(x)$ defined as\\
 $G(x_1) = x_1+\sin(x_1)-1,\\
 G(x_i) = -x_{i-1}+2x_i+\sin(x_i)-1,~~ i = 2, \cdots, n-1.\\
 G(x_n) = x_n+\sin(x_n)-1$\\
 {\bf Problem 5.7}(Problem 19 in ~\cite{la2006spectral})Zero Jacobian function $G(x)$ defined as\\$
 G(x_1)=\sum_{j=1}^nx_j^2\\
 G(x_{i}) =  -2x_1x_i,~~~for ~~i=2,\cdots, n$\\ and 
 $\Gamma =\R$\\
{\bf Problem 5.8}(Problem 14~\cite{song2024efficient}) The general interpretation of $G(x)$ defined as\\$
  G(x_1) =x_1 (x_1^2 + x_2^2) - 1\\
  G(x_i) =  x_i  (x_{i-1}^2 + 2x_i^2 + x_{i+1}^2) - 1,~~~for ~~i=2,\cdots, n-1\\
  G(x_n) = x_n  (x_{n-1}^2 + x_n^2)$\\ and 
 $\Gamma =\R$\\
 {\bf Problem 5.9}(Problem 12~\cite{la2006spectral})Trigexp function  $G(x)$ defined as\\$
  G(x_1) =3x_1^3 + 2x_2 - 5 + \sin(|x_1-x_2|)\sin(|x_1+x_2|)\\
  G(x_i) =-x_{i-1}\exp(x_{i-1}-x_i) + x_1(4+3x_i^3) + 2x_{i+1} - 5 + \sin(|x_i-x_{i+1}|)\sin(|x_i+x_{i+1}|),~~~for ~~i=2,\cdots, n-1\\
  G(x_n) =  -x_{n-1}  \exp(x_{n-1} - x_n) + 4x_n - 3$\\ and 
 $\Gamma =\R$\\
 {\bf Problem 5.10}(Problem 2~\cite{song2024efficient}) Complementary problem
 $G(x)$ defined as\\$
 G(x_i) = (x_i-1)^2 - 1.01, ~~for ~i=1, \cdots, n$\\and 
 $\Gamma =\R$\\
 {\bf Problem 5.11}(Problem 4~\cite{song2024efficient}) Complementary problem and 
 $G(x)$ defined as\\$
 G(x_i) =\frac{i}{n} \exp{x_i}-1 , ~~for ~i=1, \cdots, n$\\and 
 $\Gamma =\R$\\
{\bf Problem 5.12}(Problem 11~\cite{ibrahim2024two})\\
 $ G(x_i) =x_i- \sin(|x_i - 1|) , ~~for ~i=1, \cdots, n$\\and 
 $\Gamma =\R$\\
{\bf Problem 5.13}(Problem 4.5 in ~\cite{waziri2022two})\\$
 G(x_i) =2x_i- \sin(|x_i - 1|) , ~~for ~i=1, \cdots, n$\\and 
 $\Gamma =\R$\\
{\bf Problem 5.14}(Problem 6~\cite{song2024efficient}) 
 \\$
 G(x_i) =x_i- 2\sin(|x_i - 1|) , ~~for ~i=1, \cdots, n$\\and 
 $\Gamma =\R$\\
{\bf Problem 5.15}(Problem 11~\cite{song2024efficient}) 
 \\$
 G(x_i) =(\exp{x_i})^2+3\sin x_i\cos x_i-1 , ~~for ~i=1, \cdots, n$\\and 
 $\Gamma =\R$\\
 {\bf Problem 5.16}(Problem 5~\cite{waziri2020descent}) The  singular function
 $G(x)$ defined as\\
 $G(x_1)=2.5x_1 +x_2 - 1\\
 G(x_i) = x_{i-1} + 2.5x_i + x_{i+1} - 1, ~~for ~i=2, \cdots, n-1\\
 G(x_n)=x_{n-1} + 2.5x_n - 1$\\
 {\bf Problem 5.17}(Problem 1~\cite{zhou2007limited}) $G(x)$ defined as\\$
 G(x_i) =  2x_i - \sin(|x_i|), ~~for ~i=1, \cdots, n$\\
 {\bf Problem 5.18}(Problem 32~\cite{la2006spectral}) Minimal function  $G(x)$ is defined as\\$
 G(x_i) = 0.5\{\log{x_i}+\exp{x_i} -\sqrt{(\log{x_i}-\exp{x_i})^2-10^{-10} }\}, ~~for ~i=1, \cdots, n$\\
 {\bf Problem 5.19}(Problem 4.11~\cite{li2021scaled}) $G(x)$ defined by\\$
 G(x_i) = 2(10^{-5})(x_i-1) + 4x_i \sum_{j=1}^nx_j^2 - x_i, ~~for ~i=1, \cdots, n$\\
{\bf Problem 5.20}(Problem 4.6~\cite{sabi2023modified})  $G(x)$ defined by\\$
 G(x_i) = x_{i}(\cos(x_i-1/n)(\sin x_i-1-(1-x_i)^2-1/n\sum_{j=1}^nx_j), ~~for ~i=1, \cdots, n$

%% file: comps.tex
\subsection{Signal Restoration}
{Signal restoration refers to recovering an original signal from degraded observed signals\cite{chen2024manifold}. Signal restoration is a real-world problem that includes but is not limited to dequantization ~\cite{liu2018graph, liu2016random},  denoising ~\cite{dinesh2020point, pang2017graph, zeng20193d}, deblurring~\cite{bai2018graph, chen2021fast}. \\
Signal restoration includes large-scale inverse problems in which a multidimensional signal $x$ is to be obtained from the observation of
data $y$ consisting of signals.  Both the original signal $x$ and the observed $y$ are taken to lie in some real Hilbert spaces which may be independent~\cite{combettes2005signal, stark2013image}.\\
The observed signal is given by 
\begin{equation}\label{eq:sr}
    y=Hx+k
\end{equation}such that $x$ is the signal that we want to recover from $y$, k is called the additive noise, $H$ is a given operator representing the observation process like blurring or degradation, ~\cite{selesnick2009sparse,soussen2011bernoulli}. 
The problem of restoring the original signal $x$ from the observed signal $y$ is an inverse problem ~\cite{espanol2010multilevel, selesnick2009sparse}. 

Different techniques have been developed to resolve \eqref{eq:sr}. For example, Multilevel approach~\cite{espanol2010multilevel}, majorization-minimization\cite{selesnick2009sparse}, Accelerated Projected Gradient Method  APGM\cite{daubechies2008accelerated}  and many more. Therefore \eqref{eq:sr} can be written as
\begin{equation}\label{eq:sp}
    T(x)= ||y-Hx||_2^2+\eta||x||_1
\end{equation}
so to obtain $x$, its through minimizing/possibly solving \eqref{eq:sp} called $l_1$ regularized linear inverse problem or the  penalized least-squares problem.\\
The presence of the $l_1$ term may lead to small components
of $x$ to become exactly zero, thus leading to sparse solutions
\cite{figueiredo2007gradient}.\\
Yin et al in \cite{yin2023family}, evaluated the performance of their technique by obtaining the sparse solution $x$ for 
\begin{equation}\label{eq:23}
    \min_{x\in \R^n} \frac{1}{2}||y-Hx||^2+\eta||x||_1
\end{equation}
and \eqref{eq:23} is assumed to contain under-determined linear systems of equations such that $H\in \R^{m\times n}$ and $y\in \R^m$ and $m<<n$.

Problem \eqref{eq:23} can be redefined by  
\begin{equation}\label{eq:24}
    \min \frac{1}{2}z^T Qz+ d^Tz,~~ z\geq 0
\end{equation}
such that $z=\begin{pmatrix}
\max\{x, 0\}\\ \max\{-x, 0\}
\end{pmatrix}$, $d=\begin{pmatrix}
\eta E_n-H^Ty\\ \eta E_n+H^T y
\end{pmatrix}$, $Q=\begin{pmatrix}
H^TH&-H^TH\\ -H^TH&H^TH
\end{pmatrix}$.\\ From \cite{gao2018efficient, xiao2011non} the function $G$ is monotone and continuous.\\It was also noted in ~\cite{sabi2020two} that $z$ satisfies \eqref{eq:24} if and only if 
it satisfies \eqref{eq:application}
\begin{equation}
    \label{eq:application}
  G(z)=\min\{z,  Qz+d\}=0
\end{equation}
The 'min' is interpreted to be a point-wise minimum and $G$ is monotone and continuous as proved in ~\cite{iusem1997newton}.\\Therefore problem \eqref{eq:application} can be interpreted to be in the form of \eqref{eq:op}. In the experiment, firstly, we defined a sparse signal $x_{original}$ of length $n=2^{12}$ and sparsity $k=2^{9}$, that is $k$ are non-zero elements in the signal that are randomly selected.  \\
A Sensing Matrix $H$ of size $m\times n$ was obtained randomly. Where $m<n$ and in this case we considered $m=2^{11}$.  The Gaussian noise was also determined whose components were produced normal distributions $N(\mu=0, \sigma=0.01)$. The observation $y$ was considered to be the sum of the noise and the product of the sensing matrix and the original signal. 
\newpage

    \begin{figure}[h] 
    \centering
    \includegraphics[width=0.8\textwidth]{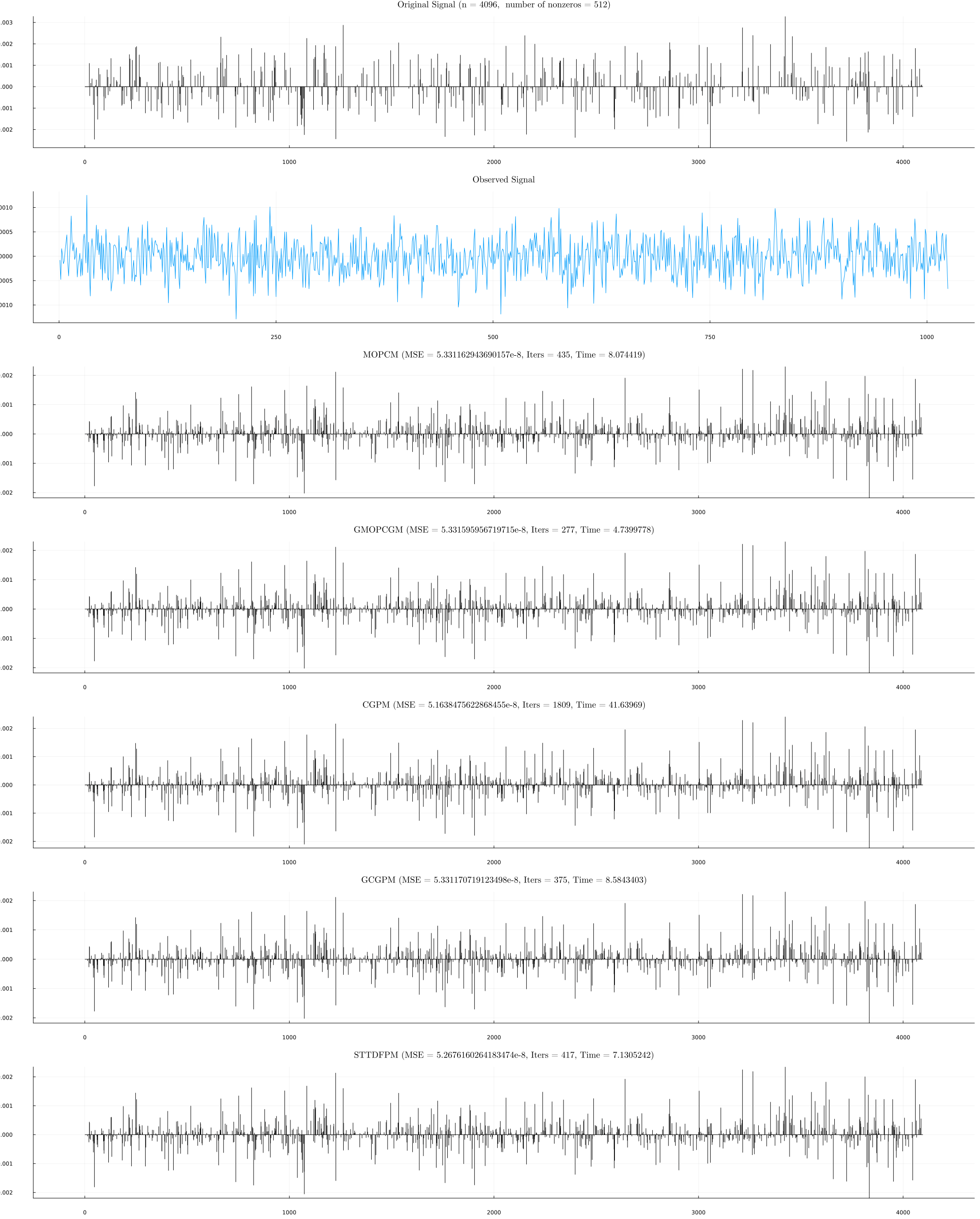} 
    \caption{Reconstructed signals}
    \label{fig:reconstructedsignals} 
\end{figure}

\newpage
In the application of the GCGPM and GMOPCGM in compressed sensing, the iteration process began with the initial point $x_o=H^Ty$ of length $n$, and the experiment was terminated when $\|G(x_k)\|< 10^{-5}$.\\
The parameters for the GCGPM remained the same.\\
For GMOPCGM, we selected the following parameters.\\
 $\tau=1.05$, $\rho=0.8$,  $\beta=0.5$,  $\zeta=0.0001$, $\alpha_{min}=0.1$, $\alpha_{max}=2.0$, $\lambda_o=1.0$, 
 $\gamma=1.1$, $\gamma_2=1.8$, $\gamma_3=0.85$, $\gamma_4=1.0$, $\zeta_1=1.0$, $\zeta_2=1.0$.

To consolidate the applicability of our schemes in recovering sparse signals, we compared them with STTDFPM in ~\cite{ibrahim2023two}, CGPM in ~\cite{zheng2020conjugate} and MOPCGM in ~\cite{sabi2023modified}. The comparison helped us to evaluate how our methods perform best relative to their counterparts. A good performance implies a small mean square error given by,
\[
MSE=\frac{\|x_{original}-x_{recovered}\|}{n}.
\]
So, for fairness to all solvers, we conducted 10 experiments and determined the means for the  Iterations, means for the function, and mean times besides the mean square errors. The results were recorded in table \ref{tab:algorithm_performance} below.
{\small 
\begin{table}[H]
    \centering
    \sisetup{scientific-notation=true, round-mode=places, round-precision=2}
    \begin{tabular}{lcccc}
        \toprule
        \textbf{Algorithm} & \textbf{Iterations} & \textbf{Function Evals } & \textbf{Time} & \textbf{MSE} \\
        \midrule
        MOPCM        & 421.7           & 1267.1 & 7.5477 & \num{5.22894E-8} \\
        GMOPCGM      & \textbf{269.0}  & \textbf{808.8}  & \textbf{4.5182} & \num{5.22928E-8} \\
        CGPM         & 1708.4          & 5468.7 & 39.6073 & \num{5.05665E-8} \\
        GCGPM        & 363.2           & 1089.6 & 8.0243 & \num{5.22875E-8} \\
        STTDFPM      & 398.2           & 831.7  & 6.8776 & \num{5.14822E-8} \\
        \bottomrule
    \end{tabular}
    \caption{Comparison of algorithm performance based on iterations, function evaluations, time, and MSE.}
    \label{tab:algorithm_performance}
\end{table}
}
As we can see from the table, it is clear that CGPM gives a small mean square error. However, GMOPCGM took less time, had few iterations, and had few function evaluations. Figure \ref{fig:reconstructedsignals} shows the quality of the reconstructed signals by the solvers.

}

%% file: concl.tex
\section*{Conclusion}
In this paper, we proposed two efficient methods, and GCGPM outperformed the rest of its counterparts. They satisfy the sufficient descent condition. They also converge globally under the Lipschitz continuity of the function.

%% file: ref.tex



\newpage

\phantomsection 
\addcontentsline{toc}{chapter}{References}
\label{references}
\printbibliography


%% file: references.bib
@article{iusem1997newton,
  title={Newton-type methods with generalized distances for constrained optimization},
  author={Iusem, N Alfredo and Solodov, V Michael},
  journal={Optimization},
  volume={41},
  number={3},
  pages={257--278},
  year={1997},
  publisher={Taylor \& Francis}
}

@article{meintjes1987methodology,
  title={A methodology for solving chemical equilibrium systems},
  author={Meintjes, Keith and Morgan, Alexander P},
  journal={Applied Mathematics and Computation},
  volume={22},
  number={4},
  pages={333--361},
  year={1987},
  publisher={Elsevier}
}

@article{zeleznik1968calculation,
  title={Calculation of complex chemical equilibria},
  author={Zeleznik, Frank J and Gordon, Sanford},
  journal={Industrial \& Engineering Chemistry},
  volume={60},
  number={6},
  pages={27--57},
  year={1968},
  publisher={ACS Publications}
}

@article{dai2020efficient,
  title={Efficient predictability of stock return volatility: The role of stock market implied volatility},
  author={Dai, Zhifeng and Zhou, Huiting and Wen, Fenghua and He, Shaoyi},
  journal={The North American Journal of Economics and Finance},
  volume={52},
  pages={101174},
  year={2020},
  publisher={Elsevier}
}

@article{figueiredo2007gradient,
  title={Gradient projection for sparse reconstruction: Application to compressed sensing and other inverse problems},
  author={Figueiredo, M{\'a}rio AT and Nowak, Robert D and Wright, Stephen J},
  journal={IEEE Journal of selected topics in signal processing},
  volume={1},
  number={4},
  pages={586--597},
  year={2007},
  publisher={IEEE}
}

@article{abubakar2020note,
  title={A note on the spectral gradient projection method for nonlinear monotone equations with applications},
  author={Abubakar, Auwal Bala and Kumam, Poom and Mohammad, Hassan},
  journal={Computational and Applied Mathematics},
  volume={39},
  number={2},
  pages={129},
  year={2020},
  publisher={Springer}
}

@article{gao2018efficient,
  title={An efficient three-term conjugate gradient method for nonlinear monotone equations with convex constraints},
  author={Gao, Peiting and He, Chuanjiang},
  journal={Calcolo},
  volume={55},
  number={4},
  pages={53},
  year={2018},
  publisher={Springer}
}

@article{sabi2023modified,
  title={Modified optimal Perry conjugate gradient method for solving system of monotone equations with applications},
  author={Sabi'u, Jamilu and Shah, Abdullah and Stanimirovi{\'c}, Predrag S and Ivanov, Branislav and Waziri, Mohammed Yusuf},
  journal={Applied Numerical Mathematics},
  volume={184},
  pages={431--445},
  year={2023},
  publisher={Elsevier}
}

@article{ibrahim2023two,
  title={Two classes of spectral three-term derivative-free method for solving nonlinear equations with application},
  author={Ibrahim, Abdulkarim Hassan and Alshahrani, Mohammed and Al-Homidan, Suliman},
  journal={Numerical Algorithms},
  pages={1--21},
  year={2023},
  publisher={Springer}
}

@article{zheng2020conjugate,
  title={A conjugate gradient projection method for solving equations with convex constraints},
  author={Zheng, Li and Yang, Lei and Liang, Yong},
  journal={Journal of Computational and Applied Mathematics},
  volume={375},
  pages={112781},
  year={2020},
  publisher={Elsevier}
}

@article{zhao2001monotonicity,
  title={Monotonicity of fixed point and normal mappings associated with variational inequality and its application},
  author={Zhao, Yun-Bin and Li, Duan},
  journal={SIAM Journal on Optimization},
  volume={11},
  number={4},
  pages={962--973},
  year={2001},
  publisher={SIAM}
}

@article{chen2019gramian,
  title={Gramian solutions and soliton interactions for a generalized (3+ 1)-dimensional variable-coefficient Kadomtsev--Petviashvili equation in a plasma or fluid},
  author={Chen, Su-Su and Tian, Bo},
  journal={Proceedings of the Royal Society A},
  volume={475},
  number={2228},
  pages={20190122},
  year={2019},
  publisher={The Royal Society Publishing}
}

@article{gao2019mathematical,
  title={Mathematical view with observational/experimental consideration on certain (2+ 1)-dimensional waves in the cosmic/laboratory dusty plasmas},
  author={Gao, Xin-Yi},
  journal={Applied Mathematics Letters},
  volume={91},
  pages={165--172},
  year={2019},
  publisher={Elsevier}
}

@article{perry1978modified,
  title={A modified conjugate gradient algorithm},
  author={Perry, Avinoam},
  journal={Operations Research},
  volume={26},
  number={6},
  pages={1073--1078},
  year={1978},
  publisher={INFORMS}
}

@article{chen2024manifold,
  title={Manifold graph signal restoration using gradient graph Laplacian regularizer},
  author={Chen, Fei and Cheung, Gene and Zhang, Xue},
  journal={IEEE Transactions on Signal Processing},
  year={2024},
  publisher={IEEE}
}

@article{espanol2010multilevel,
  title={Multilevel approach for signal restoration problems with Toeplitz matrices},
  author={Espa{\~n}ol, Malena I and Kilmer, Misha E},
  journal={SIAM Journal on Scientific Computing},
  volume={32},
  number={1},
  pages={299--319},
  year={2010},
  publisher={SIAM}
}

@article{zeng20193d,
  title={3D point cloud denoising using graph Laplacian regularization of a low dimensional manifold model},
  author={Zeng, Jin and Cheung, Gene and Ng, Michael and Pang, Jiahao and Yang, Cheng},
  journal={IEEE Transactions on Image Processing},
  volume={29},
  pages={3474--3489},
  year={2019},
  publisher={IEEE}
}

@article{pang2017graph,
  title={Graph Laplacian regularization for image denoising: Analysis in the continuous domain},
  author={Pang, Jiahao and Cheung, Gene},
  journal={IEEE Transactions on Image Processing},
  volume={26},
  number={4},
  pages={1770--1785},
  year={2017},
  publisher={IEEE}
}

@article{dinesh2020point,
  title={Point cloud denoising via feature graph laplacian regularization},
  author={Dinesh, Chinthaka and Cheung, Gene and Baji{\'c}, Ivan V},
  journal={IEEE Transactions on Image Processing},
  volume={29},
  pages={4143--4158},
  year={2020},
  publisher={IEEE}
}

@article{liu2016random,
  title={Random walk graph Laplacian-based smoothness prior for soft decoding of JPEG images},
  author={Liu, Xianming and Cheung, Gene and Wu, Xiaolin and Zhao, Debin},
  journal={IEEE Transactions on Image Processing},
  volume={26},
  number={2},
  pages={509--524},
  year={2016},
  publisher={IEEE}
}

@article{liu2018graph,
  title={Graph-based joint dequantization and contrast enhancement of poorly lit JPEG images},
  author={Liu, Xianming and Cheung, Gene and Ji, Xiangyang and Zhao, Debin and Gao, Wen},
  journal={IEEE Transactions on Image Processing},
  volume={28},
  number={3},
  pages={1205--1219},
  year={2018},
  publisher={IEEE}
}

@article{bai2018graph,
  title={Graph-based blind image deblurring from a single photograph},
  author={Bai, Yuanchao and Cheung, Gene and Liu, Xianming and Gao, Wen},
  journal={IEEE transactions on image processing},
  volume={28},
  number={3},
  pages={1404--1418},
  year={2018},
  publisher={IEEE}
}

@inproceedings{chen2021fast,
  title={Fast \& robust image interpolation using gradient graph laplacian regularizer},
  author={Chen, Fei and Cheung, Gene and Zhang, Xue},
  booktitle={2021 IEEE International Conference on Image Processing (ICIP)},
  pages={1964--1968},
  year={2021},
  organization={IEEE}
}

@book{stark2013image,
  title={Image recovery: theory and application},
  author={Stark, Henry},
  year={2013},
  publisher={Elsevier}
}

@article{combettes2005signal,
  title={Signal recovery by proximal forward-backward splitting},
  author={Combettes, Patrick L and Wajs, Val{\'e}rie R},
  journal={Multiscale modeling \& simulation},
  volume={4},
  number={4},
  pages={1168--1200},
  year={2005},
  publisher={SIAM}
}

@article{selesnick2009sparse,
  title={Sparse signal restoration},
  author={Selesnick, Ivan W},
  journal={Connexions},
  pages={1--13},
  year={2009},
  publisher={IMPR}
}

@article{soussen2011bernoulli,
  title={From Bernoulli--Gaussian deconvolution to sparse signal restoration},
  author={Soussen, Charles and Idier, J{\'e}r{\^o}me and Brie, David and Duan, Junbo},
  journal={IEEE Transactions on Signal Processing},
  volume={59},
  number={10},
  pages={4572--4584},
  year={2011},
  publisher={IEEE}
}

@article{daubechies2008accelerated,
  title={Accelerated projected gradient method for linear inverse problems with sparsity constraints},
  author={Daubechies, Ingrid and Fornasier, Massimo and Loris, Ignace},
  journal={journal of fourier analysis and applications},
  volume={14},
  pages={764--792},
  year={2008},
  publisher={Springer}
}

@article{la2006spectral,
  title={Spectral residual method without gradient information for solving large-scale nonlinear systems of equations},
  author={La Cruz, William and Mart{\'\i}nez, Jos{\'e} and Raydan, Marcos},
  journal={Mathematics of computation},
  volume={75},
  number={255},
  pages={1429--1448},
  year={2006}
}

@article{yin2023family,
  title={A family of inertial-relaxed DFPM-based algorithms for solving large-scale monotone nonlinear equations with application to sparse signal restoration},
  author={Yin, Jianghua and Jian, Jinbao and Jiang, Xianzhen and Wu, Xiaodi},
  journal={Journal of Computational and Applied Mathematics},
  volume={419},
  pages={114674},
  year={2023},
  publisher={Elsevier}
}

@article{xiao2011non,
  title={Non-smooth equations based method for L1-norm problems with applications to compressed sensing},
  author={Xiao, Yunhai and Wang, Qiuyu and Hu, Qingjie},
  journal={Nonlinear Analysis: Theory, Methods \& Applications},
  volume={74},
  number={11},
  pages={3570--3577},
  year={2011},
  publisher={Elsevier}
}

@article{he2002new,
  title={A new inexact alternating directions method for monotone variational inequalities},
  author={He, Bingsheng and Liao, Li-Zhi and Han, Deren and Yang, Hai},
  journal={Mathematical Programming},
  volume={92},
  pages={103--118},
  year={2002},
  publisher={Springer}
}

@article{malitsky2014extragradient,
  title={An extragradient algorithm for monotone variational inequalities},
  author={Malitsky, Yu V and Semenov, VV3276035},
  journal={Cybernetics and Systems Analysis},
  volume={50},
  number={2},
  pages={271--277},
  year={2014},
  publisher={Springer}
}

@article{hager2005new,
  title={A new conjugate gradient method with guaranteed descent and an efficient line search},
  author={Hager, William W and Zhang, Hongchao},
  journal={SIAM Journal on optimization},
  volume={16},
  number={1},
  pages={170--192},
  year={2005},
  publisher={SIAM}
}

@article{dai2012another,
  title={Another improved Wei--Yao--Liu nonlinear conjugate gradient method with sufficient descent property},
  author={Dai, Zhifeng and Wen, Fenghua},
  journal={Applied Mathematics and Computation},
  volume={218},
  number={14},
  pages={7421--7430},
  year={2012},
  publisher={Elsevier}
}

@article{ahookhosh2013two,
  title={Two derivative-free projection approaches for systems of large-scale nonlinear monotone equations},
  author={Ahookhosh, Masoud and Amini, Keyvan and Bahrami, Somayeh},
  journal={Numerical Algorithms},
  volume={64},
  pages={21--42},
  year={2013},
  publisher={Springer}
}

@article{chorowski2014learning,
  title={Learning understandable neural networks with nonnegative weight constraints},
  author={Chorowski, Jan and Zurada, Jacek M},
  journal={IEEE transactions on neural networks and learning systems},
  volume={26},
  number={1},
  pages={62--69},
  year={2014},
  publisher={IEEE}
}

@article{candes2015phase,
  title={Phase retrieval via Wirtinger flow: Theory and algorithms},
  author={Candes, Emmanuel J and Li, Xiaodong and Soltanolkotabi, Mahdi},
  journal={IEEE Transactions on Information Theory},
  volume={61},
  number={4},
  pages={1985--2007},
  year={2015},
  publisher={IEEE}
}

@article{berry2007algorithms,
  title={Algorithms and applications for approximate nonnegative matrix factorization},
  author={Berry, Michael W and Browne, Murray and Langville, Amy N and Pauca, V Paul and Plemmons, Robert J},
  journal={Computational statistics \& data analysis},
  volume={52},
  number={1},
  pages={155--173},
  year={2007},
  publisher={Elsevier}
}

@article{sabi2020two,
  title={Two optimal Hager-Zhang conjugate gradient methods for solving monotone nonlinear equations},
  author={Sabi'u, Jamilu and Shah, Abdullah and Waziri, Mohammed Yusuf},
  journal={Applied Numerical Mathematics},
  volume={153},
  pages={217--233},
  year={2020},
  publisher={Elsevier}
}

@article{zhou2007limited,
  title={Limited memory BFGS method for nonlinear monotone equations},
  author={Zhou, Weijun and Li, Donghui},
  journal={Journal of Computational Mathematics},
  pages={89--96},
  year={2007},
  publisher={JSTOR}
}

@article{song2024efficient,
  title={An efficient inertial subspace minimization CG algorithm with convergence rate analysis for constrained nonlinear monotone equations},
  author={Song, Taiyong and Liu, Zexian},
  journal={Journal of Computational and Applied Mathematics},
  volume={446},
  pages={115873},
  year={2024},
  publisher={Elsevier}
}

@article{li2021scaled,
  title={Scaled three-term derivative-free methods for solving large-scale nonlinear monotone equations},
  author={Li, Qun and Zheng, Bing},
  journal={Numerical Algorithms},
  volume={87},
  number={3},
  pages={1343--1367},
  year={2021},
  publisher={Springer}
}

@article{ibrahim2024two,
  title={Two-step inertial derivative-free projection method for solving nonlinear equations with application},
  author={Ibrahim, Abdulkarim Hassan and Al-Homidan, Suliman},
  journal={Journal of Computational and Applied Mathematics},
  pages={116071},
  year={2024},
  publisher={Elsevier}
}

@article{waziri2022two,
  title={Two descent Dai-Yuan conjugate gradient methods for systems of monotone nonlinear equations},
  author={Waziri, Mohammed Yusuf and Ahmed, Kabiru},
  journal={Journal of Scientific Computing},
  volume={90},
  pages={1--53},
  year={2022},
  publisher={Springer}
}

@article{waziri2020descent,
  title={Descent Perry conjugate gradient methods for systems of monotone nonlinear equations},
  author={Waziri, Mohammed Yusuf and Hungu, Kabiru Ahmed and Sabi’u, Jamilu},
  journal={Numerical Algorithms},
  volume={85},
  pages={763--785},
  year={2020},
  publisher={Springer}
}

@book{nocedal1999numerical,
  title={Numerical optimization},
  author={Nocedal, Jorge and Wright, Stephen J},
  year={1999},
  publisher={Springer}
}

@article{dolan2002benchmarking,
  title={Benchmarking optimization software with performance profiles},
  author={Dolan, Elizabeth D and Mor{\'e}, Jorge J},
  journal={Mathematical programming},
  volume={91},
  pages={201--213},
  year={2002},
  publisher={Springer}
}
